\RequirePackage{fix-cm}
\documentclass[smallcondensed]{svjour3}
\smartqed  
\usepackage{graphicx}
\usepackage{bm}
\usepackage{amsfonts}
\usepackage{amsmath} 
\usepackage{amssymb} 
\usepackage{subfigure} 
\usepackage{verbatim}
\usepackage{algorithm}
\usepackage{algorithmic}
\usepackage{color} 
\usepackage{pgf} 
\usepackage{tikz}
\usetikzlibrary{arrows,automata}
\usetikzlibrary{shapes.multipart}
\usepackage[latin1]{inputenc}
\usepackage[abs]{overpic} 
\usepackage[round,comma]{natbib}
\allowdisplaybreaks[1] 
\graphicspath{{./figures/}{./}}

\newcommand{\Rtrue}{R^{\textrm{\rm true}}}
\newcommand{\Rempirical}{R^{\textrm{\rm emp}}}
\newcommand{\Cbudget}{C_{\textrm{\rm budget}}}
\newcommand{\abudget}{a_{\textrm{\rm budget}}}
\newcommand{\azero}{a_0}
\newcommand{\gprime}{ \frac{e^{B_{b}X_{b}}}{(1+e^{B_{b}X_{b}})^2}}
\newcommand{\dist}{\textrm{dist}}
\newcommand{\alphaa}{ \frac{1}{2} + \frac{\|\abudget\|_{2}^{-1}+\frac{\epsilon}{32X_{b}}}{B_{b} + \frac{\epsilon}{32X_{b}}}  \frac{\Gamma\left[1+\frac{d}{2}\right]}{\sqrt{\pi}\Gamma\left[\frac{d+1}{2}\right]} {}_{2}F_{1}\left(\tfrac{1}{2},\tfrac{1-d}{2};\tfrac{3}{2};\left(\tfrac{\|\abudget\|_{2}^{-1}+\frac{\epsilon}{32X_{b}}}{B_{b} + \frac{\epsilon}{32X_{b}}}\right)^{2}\right)}
\newcommand{\alphab}{1 - \frac{1}{2} I_{1-\left(\|\abudget\|_{2}^{-1}+\frac{\epsilon}{32X_{b}}\right)^{2}/\left(B_{b}+\frac{\epsilon}{32X_{b}}\right)^{2}} \left(\frac{d+1}{2}, \frac{1}{2} \right)}

\newcommand{\one}{\textbf{1}}
\newcommand{\R}{\mathbb{R}}
\newcommand{\X}{\mathcal{X}}
\newcommand{\F}{\mathcal{F}}
\newcommand{\Y}{\mathcal{Y}}

\newcommand{\Obj}{\textrm{\rm OpCost}}
\newcommand{\argminpi}{\hbox{ \raise-1.6mm\hbox{$\textstyle
    \mathrm{argmin} \atop \pi\in\Pi$}}}
\newcommand{\argminf}{\hbox{ \raise-1.6mm\hbox{$\textstyle
    \mathrm{argmin} \atop f\in\F^{unc}$}}}
\newcommand{\minpi}{\hbox{ \raise-1.6mm\hbox{$\textstyle
    \mathrm{min} \atop \pi\in\Pi$}}}    
\newcommand{\argmin}{\textrm{argmin}}
\newcommand{\cuppi}{\hbox{ \raise-1.6mm\hbox{$\textstyle
    \cap \atop g\in \F_{good}$}}}
\newcommand{\cuuppi}{\hbox{ \raise-1.6mm\hbox{$\textstyle
    \cup \atop g\in \F_{good}$}}} 
\newcommand{\gprimeCtwo}{\frac{e^{-B_{b}X_{b}}}{1+e^{-B_{b}X_{b}}}}

\journalname{Mach Learn}
\begin{document}

\title{On Combining Machine Learning with Decision Making\thanks{Funding for Theja Tulabandhula was provided by a Fulbright Fellowship and Xerox Fellowship. Cynthia Rudin's work on this project was funded in part by Con Edison, by the MIT Energy Initiative Seed Fund, and NSF grant IIS-1053407.}
}

\author{Theja Tulabandhula         \and
         Cynthia Rudin
}

\institute{Theja Tulabandhula \at
               Department of Electrical Engineering and Computer Science,\\ 
               Massachusetts Institute of Technology, Cambridge, MA 02139, USA.\\
              \email{theja@mit.edu}
           \and
           Cynthia Rudin \at
              MIT Sloan School of Management,\\
       Massachusetts Institute of Technology, Cambridge, MA 02139, USA.\\
       \email{rudin@mit.edu}
}

\date{Received: date / Accepted: date}
\maketitle
\begin{abstract}
We present a new application and covering number bound for the framework of ``Machine Learning with Operational Costs (MLOC)," which is an exploratory form of decision theory. The MLOC framework incorporates knowledge about how a predictive model will be used for a subsequent task, thus combining machine learning with the decision that is made afterwards. In this work, we use the MLOC framework to study a problem that has implications for power grid reliability and maintenance, called the \textit{Machine Learning and Traveling Repairman Problem} (ML\&TRP). The goal of the ML\&TRP is to determine a route for a ``repair crew," which repairs nodes on a graph. The repair crew aims to minimize the cost of failures at the nodes, but as in many real situations, the failure probabilities are not known and must be estimated. The MLOC framework allows us to understand how this uncertainty influences the repair route. We also present new covering number generalization bounds for the MLOC framework.
\end{abstract}

\section{Introduction}
In many domains, it is essential to understand how uncertainty in predictions influences decision-making. In that sense, one would like to explore the space of possible reasonable predictions and understand the range of reasonable policies and their costs. The new framework of Machine Learning with Operational Costs (MLOC) \citep[][]{TulabandhulaRu11ArXiv} provides a mechanism to do this, and is a type of exploratory decision theory. Where usual decision theories  provide a single policy that minimizes expected costs, the MLOC framework is able to produce a range of reasonable policies that span the full set of reasonable costs. To do this, the operational cost becomes a regularization term within the machine learning model, and adjusting the regularization constant allows us to explore solutions for all reasonable costs. This gives decision makers a way to understand the uncertainty in their predictive model in terms of something they can grasp - uncertainty in the cost to solve the problem.

The MLOC framework can also be used in another way, namely to incorporate prior knowledge about the cost to produce a better predictive model. In that sense, knowledge about the cost translates into a more restricted hypothesis space, which potentially translates into better generalization. In particular, if the hypothesis space is restricted, then upper bounds on the complexity of the hypothesis space are smaller, leading to better generalization bounds.

In this work, we provide an application of the MLOC framework to power grid engineering and reliability. This problem, called the \textit{Machine Learning and Traveling Repairman Problem} (ML\&TRP), has a machine learning component and a decision-making component. The machine learning component is to predict future power grid failures before they occur, where these failures occur at equipment that is distributed throughout the city. The decision-making component is to determine in what order the equipment should be inspected. We could use the MLOC framework in either of the two ways outlined above: either to understand the range of reasonable costs for the power company, or to use prior knowledge that the costs are high or low in order to choose a more predictive and cost-effective route. 

To be more precise, the ML\&TRP \textit{prediction} problem is to determine the failure probability for each node on a graph, using features of each node and past failure data. The \textit{decision} problem is to determine a route for a ``repair crew" on the graph, where there is some travel time between each pair of nodes. There are many possible applications of the ML\&TRP, including the scheduling of safety inspections or repair work for the electrical grid, oil rigs, underground mining, machines in a factory, or airlines. In our experiments, we use data from an ongoing project with Con Edison, which is NYC's power utility company.

We also provide a generalization bound for the MLOC framework based on covering numbers. These bounds are different than those of \cite{TulabandhulaRu11ArXiv} which use concentration of Rademacher complexity and Dudley's entropy integral, and are not directly comparable. The bounds here have a much more geometric flavor looking at the hypothesis space as a volumetric object. Neither of the two bounds are tighter in all situations. 
We find the bounds here to be more intuitive, as the geometry is more transparent.

The ML\&TRP relates to literature on both machine learning and optimization (time-dependent traveling salesman problems). In machine learning, 
our work bears a slight resemblance to work on graph-based regularization \citep[]{agarwal06-graph-ranking,MR_JMLR_06,ZhouETAL04}, but their goal is to obtain probability estimates that are smoothed on a graph with suitably designed edge weights. On the other hand, our goal is to obtain, in addition to probability estimates, a low-cost route for traversing a very different graph with edge weights that are physical distances. Our regularization is vastly different from popular ones ($\ell_{1}$ or $\ell_{2}$ norm) because our regularization comes from beliefs on decision-making costs. We use unlabeled data as does semi-supervised learning \citep[][]{ChaSchZie06} but differ in the motivation as well as the way we use these additional data. For example, we do not extract distributional information from the unlabeled data. Our work contributes to the literature on the TRP (Traveling Repairman Problem) and related problems by adding the new dimension of probabilistic estimation at the nodes. We create new adaptations of modern techniques \citep{fischetti93,eijl95,lechmann09} within our work for solving the TRP part of the ML\&TRP.

There is a body of literature regarding cost models for maintenance in the reliability modeling literature, though the emphasis in those works is usually to design a model that accurately represents the stochastic process for the failures. In that literature, for instance, a maintenance schedule would be created from the predicted condition of the equipment (but not on the cost of performing the repairs in a certain order or routing a vehicle between the equipment).
\citet{barbera96} develop a model that assumes that equipment have exponential rates of failure and fail only once in an inspection interval, and they use this model to determine a maintenance schedule. \citet{marseguerra02} introduces a model for degradation leading to failure for a continuous complex system, and use Monte Carlo simulations to determine the optimal degradation level to perform an inspection. Their work uses a very different cost model from ours; the cost is the long run average maintenance cost and cost of failures. A neural-network based maintenance model was developed by \citet{heng09}. 
A related work on routing for emergency maintenance on the electrical grid is the heuristic algorithm of \citet{weintraub99} that dispatches vehicles to areas where there are currently breakdowns and where there are likely to be breakdowns in the future. \citet{ErtekinRuMcAAAI2013} propose a model for failures of power grid equipment and use this model to simulate the cost of various inspection policies. 

One can view the MLOC framework to be somewhat Bayesian, in the sense that prior knowledge is being used when not enough data are available.

In Section \ref{sec:MLOC} we review the MLOC framework. In Section \ref{sec:formulations} we will motivate and outline the new application of the MLOC framework to the ML\&TRP, providing two ways of modeling failure cost. In Section \ref{sec:minlps} we provide mixed-integer nonlinear (MINLP) formulations and discuss algorithms an illustrative example. Section \ref{sec:grid} gives experimental results on data from the NYC power grid, showing the benefit of the ML\&TRP over traditional methods. Section \ref{sec:generalizationbound} contains the theoretical generalization result for the MLOC framework with proofs. Section \ref{sec:conclusion} concludes the paper. 
The conference paper of \citep{TulabandhulaRuJaADT11} contains a summary of work on the ML\&TRP, and the paper \cite{TulabandhulaRu11ArXiv} provides a more complete explanation of the MLOC framework, with other illustrations and connections to robust optimization.

\section{Review of Framework for Machine Learning with Operational Costs}\label{sec:MLOC}

In the MLOC framework we have the standard supervised training set of labeled instances, $\{(x_i, y_i)\}_{i=1}^m$, where $x_i\in\X$, $y_i\in\Y$. For simplicity, $\X\subset \mathbb{R}^{d}$. To have nonlinear functions, we could simply have the $j^{\textrm{th}}$ component of $x$ replaced by a nonlinear function $h_j(x)$. Also $\Y\subset \mathbb{R}$. We wish to learn a function $f^*:\X\rightarrow \Y$. This is ordinarily done by solving a minimization problem:
\begin{equation}
f^*\in\argmin_{f\in\F^{unc}}\left( \sum_{i=1}^{m} l(f(x_i),y_i)+C_{2}R(f)\right),
\label{eqn:reg-train-loss}
\end{equation}
for some loss function $l:\Y\times\Y \rightarrow \mathbb{R}_{+}$, regularizer $R:\F^{unc} \rightarrow \mathbb{R}$, constant $C_{2}$ and function class $\F^{unc}$. 
$\F^{unc}$ is the set of all linear functionals, where $f \in \mathcal{F}^{unc}$ is of the form $\lambda \cdot x$, $\lambda \in \mathbb{R}^{d}$. The superscript `$unc$' refers to the word ``unconstrained.'' 

Consider an organization making a policy decision regarding a new collection of unlabeled instances $\{\tilde{x}_i\}_{i=1}^M \in \mathcal{X}^{M}$. The cost to enact a policy is not exactly known, because the labels for the $\{\tilde{x}_i\}_{i}$ are not known. Instead the model's predictions are used, which are the $f^*(\tilde{x}_i)$'s. The goal of the organization is then to create a policy $\pi^*$ that minimizes operational cost $\Obj(\pi,f^*,\{\tilde{x}_i\}_i)$. The operational cost $\Obj(\pi,f^*,\{\tilde{x}_{i}\}_i)$ is how much will be spent if policy $\pi$ is chosen in response to the $\{f^*(\tilde{x}_{i})\}_{i}$'s.
When there is uncertainty in $f^*$, there is uncertainty in the cost to enact the optimal policy $\pi^*$. This uncertainty is what we would like to explore.
A typical way that companies make decisions is using what we call the \textbf{sequential process}, which computes the policy according to two steps: 
\textcolor{black}{
\begin{description}
\item [Step 1:] Create function $f^{*}$ based on $\{(x_i, y_i)\}_i$ according to (\ref{eqn:reg-train-loss}). That is: 
\[f^*\in\argmin_{f\in\F^{unc}} \left(\sum_{i=1}^{m} l(f(x_i),y_i)+C_{2}R(f)\right).\]
\item [Step 2:] Choose policy $\pi^*$ to minimize the operational cost, 
\[\pi^*\in\argmin_{\pi\in\Pi} \Obj(\pi,f^*,\{\tilde{x}_{i}\}_i).\] 
\end{description}
}
On the other hand, the MLOC framework is based around a \textbf{simultaneous process}, which combines Steps 1 and 2 of the sequential process. To do this, the operational cost becomes a regularization term, and its regularization parameter $C_1$ controls the amount of optimism or pessimism for the operational cost.
\begin{description}
\item [Step 1:]  Choose a model $f^*$ obeying the following:
\begin{align}\nonumber
f^* \in \argminf \left[ \sum_{i=1}^{m} l\left(f(x_i),y_i\right) 
+C_2 R(f) +C_1 \minpi \Obj\left(\pi,f,\{\tilde{x}_{i}\}_i\right) \right].\hspace*{0pt}\label{eqn:optimisticbias}
\end{align}
\item [Step 2:] Compute the policy: $$\pi^* \in \argminpi \Obj\left(\pi,f^*,\{\tilde{x}_{i}\}_i\right).$$
\end{description}
\textcolor{black}{
The case $C_1=0$ for the simultaneous process is precisely the sequential process; thus, the sequential process is a special case of the simultaneous process. 
}
Our ability to solve the MLOC simultaneous process depends on the tractability of the optimization problem $\argminpi \Obj\left(\pi,f^*,\{\tilde{x}_{i}\}_i\right)$. However, if this problem is intractable, then the sequential process is also intractable, and the organization will not be able to choose an optimized policy at all. The simultaneous process requires this subproblem to be solved several times, whereas the sequential process only requires the subproblem to be solved once. If the number of unlabeled instances is small, then Step 1 can be solved without a problem, even if the training set is large.
As $C_1$ varies over its full range, it maps out the full range of costs for all reasonable solutions. If $C_1$ is set to a number that is too large (either positive or negative), the solution of the simultaneous process will have empirical error that is too high to be reasonable. In that case, we know that by varying $C_1$ within a smaller range will lead to the full range of costs for reasonable predictive models.

As with any regularization term, the new operational cost term can be interpreted as a prior belief about the model - in this case, a belief that the operating costs should be lower or higher on the current set of unlabeled instances $\{\tilde{x}_{i}\}_{i}$. In that sense, MLOC regularization may have a closer connection to reality than typical (e.g., $\ell_{1}$ or $\ell_{2}$ norm) regularizers. If one asks a manager at a company what prior belief they have about the estimation model, it is not likely they would give a answer in terms of coefficients for a linear model. Even managers who are not mathematicians or computer scientists might have some belief - they could perhaps believe that they are expecting to spend a certain amount to enact the policy. It is possible that this type of belief, which relies on direct experience, might be more practical, and more accurate, than the more abstract prior information that we are typically used to dealing with. 
In the ML\&TRP, the training error term is derived from data from the past, and the OpCost term is calculated on data from the present. The OpCost term is the only term that deals with routing. 

\section{The Machine Learning and Traveling Repairman Problem}\label{sec:formulations}

The US Department of Energy's Grid 2030 document states that ``America's electric system, `the supreme engineering achievement of the 20th century,' is aging, inefficient, and congested, and incapable of meeting the future energy needs of the Information Economy without operational changes and substantial capital investment over the next several decades"  \citep{GRID2030}.
Since 2004, many power utility companies are implementing new inspection and repair programs for preemptive maintenance, whereas in the past, all repair work was done reactively \citep{Urbina04}. New York City has the oldest power system in the world, and the largest underground electric system, with enough electrical cable to go three and a half times around the world. In New York City, there are several separate new preemptive maintenance programs, including the targeted inspection program for electrical service structures (manholes), programs that perform extensive repairs that were placed on a waiting list after the manhole was inspected, and the \textit{vented cover replacement program}, where each manhole is replaced with a vented cover that allows gases to escape, mitigating the possibility and effects of serious events including fires and explosions. Con Edison, the power company in NYC, has the ability to use machine learning models in Manhattan, Brooklyn and the Bronx for scheduling of manhole inspection and repair work 
\citep{RudinETAL10,RudinETAL11,RudinEtAl2011ComputerMagazine,RudinEtAl14}.
This project was the motivation for the development of the ML\&TRP and we use data from the NYC power grid for our experiments. Features for the NYC model are derived from physical characteristics of the manhole (e.g., number of electrical cables entering the manhole), and features derived from its history of involvement in past events. Repeat failures (serious and non-serious events) can occur on the same manhole. We take the possibility of repeat failures into account in the ML\&TRP (in Cost 1 given below).
That said, failures are rare events, and it is not easy to accurately estimate the probability that a given manhole will fail within a given period of time. Because of this uncertainty, we can use the MLOC framework to assist in decision-making. 
The result $\pi^{*}\in\Pi$ from the algorithm would be a route that could be used for the repair crew to fix a pre-specified set of manholes corresponding to $\{\tilde{x}_i\}_{i=1}^M$, which are assumed to need a particular repair.

\subsection{Learning}
In what follows, we will use descriptions and terminology that match the power grid application.
In the ML\&TRP, data from the past will be used to train the model, denoted $\{(x_i,y_i)\}_{i=1}^m$, whereas the $\tilde{x}_i$ are calculated from the present, whose labels are from the future and thus not known. Let $x_{i}^{j}$ indicate the $j$-th coordinate of the feature vector for manhole $i$ calculated at a time period from the past. The $x_i$ vector encodes the number and types of electrical cables, number and types of previous events, etc. The label for manhole $i$ from the past is denoted $y_i$, where $y_i\in\{-1,1\}$ indicating whether the manhole had a failure (fire, explosion, smoking manhole) within a specific period of time in the past. 
More details about the features and labels can be found in Section \ref{sec:grid}. 
The other instances $\{\tilde{x}_{i}\}_{i=1}^{M}$ (with $M$ unrelated to $m$), are unlabeled data that are each associated with a node on a graph $G$. The nodes of the graph $G$ indexed by $i=1,...,M$ represent manholes on which we want to design a route.  Note that $M$ can be substantially smaller than $m$, e.g., $M<10$ and $m>20,000$; e.g., for a repair truck that carries supplies for at most $M$ repairs. We are also given physical distances $d_{i,j} \in \mathbb{R}_+$ between each pair of nodes $i$ and $j$. A route on $G$ is represented by a permutation $\pi$ of the node indices $1,\ldots,M$. Let $\Pi$ be the set of all permutations of $\{1,...,M\}$. Failure probabilities will be estimated at each of the nodes and these estimates will be based on a function of the form $f_{\lambda}(x) = \lambda \cdot x$. 
The class of possible functions $\mathcal{F}$ is chosen to be: $ \mathcal{F}:= \{f_{\lambda} : \lambda\in\R^d ,  \|\lambda\|_{2} \leq B_{b}\} $, where $B_{b}$ is a fixed positive real number.  We choose the {logistic loss}: $l(f_{\lambda}(x),y):=\ln\left(1+e^{-yf_{\lambda}(x)}\right)$ so that {the probability of failure} $P(y=1|x)$,  is estimated as in logistic regression by:
\begin{equation}\label{probs}
{P}(y=1|x) \textrm{ or } p(x):=\frac{1}{1+e^{-f_{\lambda}(x)}}.
\end{equation}

Note that the routing problem is done in batch: once the route is determined, the repair truck is sent out and changes to the route are no longer possible.

\subsection{Two Options for the OpCost}\label{subsec:2models}
The operational cost can be defined to match the application. In the first option (denoted as Cost 1),  for each node there is a cost for (possibly repeated) failures prior to a visit by the repair crew. In this case, temporary repairs are made to fix each node before the repair crew comes to make permanent repairs. In the second option (denoted as Cost 2), for each node, there is a cost for the first failure prior to visiting it. In this case, permanent repairs are made when there is an event, or when the repair crew arrives, whichever is sooner. There is a natural interpretation of the failures as being generated by a continuous random process at each of the nodes. When discretized in time, this is approximated by a Bernoulli process with parameter $p(\tilde{x}_{i})$. Both Cost 1 and Cost 2 are appropriate for power grid applications. Cost 2 is also appropriate for delivery truck routing applications, where perishable items can fail (once an item has spoiled, it cannot spoil again). 

For convenience, we assume that after the repair crew visits all the nodes, it returns to the starting node (node $1$) which is fixed beforehand. Scenarios where one is not interested in beginning from or returning to the starting node would be modeled slightly differently (the computational complexity remains the same). {Let a route be represented by $\pi:\{1,...,M\}\mapsto \{1,...,M\} $,} this means that $\pi(i)$ is the $i^{\textrm{th}}$ node to be visited. For example, let $M=4, \pi = [2,3,4,1]$. This means, $\pi(1) = 2,$ node 2 is the first node to be visited, $\pi(2) = 3$, node 3 is the second node on the route, and so on. Since the final node visited is the first node, we append the following to the definition of $\pi$: $\pi(M+1) = \pi(1)$. Let the distances be scaled appropriately so that a unit of distance is traversed in a unit of time. Given a route, the \textit{latency} of a node $\pi(i)$ is the time (or equivalently distance) from the start at which node $\pi(i)$ is visited. It is the sum of distances traversed before position $i$ on the route:
\begin{equation}\label{latency}
L_{\pi}(\pi(i)) :=
\left\{
\begin{array}{ll}
\sum_{k=1}^{M} d_{\pi(k) \pi(k+1)}  \one_{[k< i]} & i=2,...,M \vspace{3pt} \\
\sum_{k=1}^{M} d_{\pi(k) \pi(k+1)} & i=1.
\end{array}
\right.
\end{equation}
\textcolor{black}{The starting node $\pi(1)$ thus has a latency $L_{\pi}(\pi(1))$ which is the total length of the route starting at node $\pi(1)$ and ending at node $\pi(1)$ after visiting all other nodes. 
}
\subsubsection*{Cost 1: Cost is Proportional to Expected Number of Failures Before the Visit}

Up to the time that node $\pi(i)$ is visited by the repair crew, there is a probability $p(\tilde{x}_{\pi(i)})$ that a failure will occur within each unit time interval. \textcolor{black}{Equivalently, within each unit time interval, failures are determined by a Bernoulli random variable with parameter $p(\tilde{x}_{\pi(i)})$.} 
Thus, in a time interval of length $L_{\pi}(\pi(i))$ units, the number of node failures follows the binomial distribution $\textrm{Bin}\left(L_{\pi}(\pi(i)),p(\tilde{x}_{\pi(i)})\right)$.
\textcolor{black}{For each node, we will associate a cost proportional to the expected number of failures before the repair crew's visit, as follows:}
\begin{eqnarray}
\nonumber \textrm{Cost of node $\pi(i)$} &\propto& E(\textrm{number failures in}\; L_{\pi}(\pi(i))\; \textrm{time units}) \\&=& \textrm{mean of Bin}(L_{\pi}(\pi(i)),p(\tilde{x}_{\pi(i)})) = p(\tilde{x}_{\pi(i)})L_{\pi}(\pi(i)). \label{eqn:model1costLi}
\end{eqnarray}
Using this cost, if the failure probability for node $\pi(i)$ is small, we can afford to visit it later on, trading off its latency $L_{\pi}(\pi(i))$. 
\textcolor{black}{
If $p(\tilde{x}_{\pi(i)})$ {is large}, we should visit node $\pi(i)$ earlier to keep our overall failure cost low.
} 
The failure cost of route $\pi$ is then $\textrm{OpCost}(\pi,f_{\lambda},\{\tilde{x}_{i}\}_{i=1}^M,\{d_{i,j}\}_{i,j=1}^M)= \sum_{i=1}^M p(\tilde{x}_{\pi(i)}) L_{\pi}(\pi(i))$.

\textcolor{black}{
Substituting the definition of $L_{\pi}(\pi(i))$ from (\ref{latency}):
\begin{eqnarray}\nonumber
\lefteqn{\textrm{OpCost}(\pi,f_{\lambda},\{\tilde{x}_{i}\}_{i=1}^M,\{d_{i,j}\}_{i,j=1}^M)=}\\
&& \sum_{i=2}^M p(\tilde{x}_{\pi(i)})\sum_{k=1}^M d_{\pi(k) \pi(k+1)}  \one_{[k < i]} + p(\tilde{x}_{\pi(1)})\sum_{k=1}^M d_{\pi(k) \pi(k+1)},
\label{cost1}
\end{eqnarray}
where $p(\tilde{x}_{\pi(i)})$ is given in (\ref{probs}). This will be Cost 1.
There are ways to make Cost 1 more general. The individual node cost in (\ref{eqn:model1costLi}) assumes that the node's failure probability $p(\tilde{x}_{\pi(i)})$ becomes zero after the repair crew's visit, so that for the remainder of the route, the cost incurred at this node is $\propto 0 \times (L_{\pi}(\pi(1))-L_{\pi}(\pi(i))$. We could relax this by assuming $p(\tilde{x}_{\pi(i)})$ does not vanish after the repair crew's visit and adding an additional cost for the expected failures in this period. That is, if $\beta$ is a constant of proportionality for the cost after visiting node $\pi(i)$, then the cost would become:
\begin{eqnarray*}
\textrm{Cost of node }\pi(i) = \beta\left[L_{\pi}(\pi(1)) - L_{\pi}(\pi(i))\right]p(\tilde{x}_{\pi(i)}) + L_{\pi}(\pi(i)) p(\tilde{x}_{\pi(i)}).
\end{eqnarray*}
If $\beta=1$, then the repair crew does not have any effect and cost of each node is independent of its expected number of failures before the repair crew's visit. Typically, we expect that the repair crew will repair the node so that it will not fail, and the second term above is much larger than the first.  Taking the constant of proportionality as $\beta=0$, we return to the individual costs given by (\ref{eqn:model1costLi}).
}

Note that since the cost is a sum of $M$ terms, it is invariant to ordering or indexing (caused by $\pi$). Thus we can rewrite the cost as
\begin{equation}\label{cost1def}
\textrm{OpCost}(\pi,f_{\lambda},\{\tilde{x}_{i}\}_{i=1}^M,\{d_{i,j}\}_{i,j=1}^M)= \sum_{i=1}^M p(\tilde{x}_{i}) L_{\pi}(i).
\end{equation}

 \subsubsection*{Cost 2: Cost is Proportional to Probability that the First Failure is Before the Visit}
 \label{subsec:model2description}
This cost reflects the penalty for not visiting a node before the first failure occurs there. 
This model is governed by the geometric distribution. 
\textcolor{black}{
Let the parameter of the distribution be $p$.
Then the probability that the first failure for node $\pi(i)$ occurs at time index $t>0$  is $p(1-p)^{t-1}$. The probability that the first failure for node $\pi(i)$ occurs before time $L_{\pi}(\pi(i))$ is then the sum of the failure probabilities from $t=1,...,L_{\pi}(\pi(i))$ : $\sum_{t=1}^{L_{\pi}(\pi(i))}p(1-p)^{t-1} = 1-(1-p)^{L_{\pi}(\pi(i))}$. Thus, substituting the expression (\ref{probs}) for $p$, we have:
}
\begin{eqnarray*}
\lefteqn{{P}\Big(\textrm{first failure occurs before time } L_{\pi}(\pi(i))\Big) = 1 - (1-p(\tilde{x}_{\pi(i)}))^{L_{\pi}(\pi(i))}}\nonumber\\
&=& 1 - \left(1 - \frac{1}{1+ e^{-f_{\lambda}(\tilde{x}_{\pi(i)})}}\right)^{L_{\pi}(\pi(i))}
=1 - \left(1+ e^{f_{\lambda}(\tilde{x}_{\pi(i)})}\right)^{-L_{\pi}(\pi(i))}.
\end{eqnarray*}

\textcolor{black}{
The cost of visiting node $\pi(i)$ will be proportional to this quantity:  
\begin{eqnarray}
\textrm{Cost of node } \pi(i) 
&\propto&  \left(1 - \left(1+ e^{f_{\lambda}(\tilde{x}_{\pi(i)})}\right)^{-L_{\pi}(\pi(i))}\right). 
\label{model2costLi}
\end{eqnarray}
 Similarly to Cost 1, $L_{\pi}(\pi(i))$ influences the cost at each node. If we visit a node early in the route, then the cost incurred is small because the node is less likely to fail before we reach it. Similarly, if we schedule a visit later on in the tour, the cost is higher because the node has a higher chance of failing prior to the repair crew's visit. 
}

The total failure cost is thus:
\begin{equation}\label{Cost2}
\textcolor{black}{ \textrm{OpCost}(\pi,f_{\lambda},\{\tilde{x}_{i}\}_{i=1}^M,\{d_{i,j}\}_{i,j=1}^M)= }
\sum_{i=1}^M \left(1 - \left(1+ e^{f_{\lambda}(\tilde{x}_{\pi(i)})}\right)^{-L_{\pi}(\pi(i))}\right).
\end{equation}
This cost is not directly related to a weighted TRP cost in its present form. That is, when the failure probabilities of the nodes are all the same, the total cost is not linear in the latencies, as is the case for Cost 1. Building on this cost, we will derive a cost that is the same as a weighted  TRP in Section \ref{subsec:jointminimize_model2}, of the form: 
\begin{eqnarray} \label{alternative}
\textrm{Cost of node } \pi(i) &\propto&
L_{\pi}(\pi(i))\log \left(1+e^{f_{\lambda}(\tilde{x}_{\pi(i)})}\right),
\end{eqnarray}
as an alternative to (\ref{model2costLi}). 

\textcolor{black}{
 There is a slightly more general version of this formulation (as there was for Cost 1), which is to take the cost for each node to be a function of two quantities: the probability of failure before the visit, and the probability of failure after the visit. Let us redefine $\beta$ to be a constant of proportionality for the cost of visiting before the failure event. From the geometric distribution, ${P}$(failure occurs after time  $L_{\pi}(\pi(i))) = (1-p(\tilde{x}_{\pi(i)}))^{L_{\pi}(\pi(i))},$ and the cost of visiting node $\pi(i)$ becomes:
 \begin{eqnarray*}
 \textrm{Cost of node } \pi(i) &\propto& P(\textrm{failure before } L_{\pi}(\pi(i))) + \beta \times P\left(\textrm{failure after } L_{\pi}(\pi(i))\right).
 \end{eqnarray*}
 If $\beta = 1$, then the sum above is $1$ for all nodes regardless of node failures or latencies. 
 More realistically, the cost of visiting the node after the failure is more than the cost of visiting proactively, $\beta\ll 1$ leading to (\ref{model2costLi}). 
 }
 \textcolor{black}{
 We could again have written the summation to hide the dependence on $\pi$:  
 \[\textrm{OpCost}(\pi,f_{\lambda},\{\tilde{x}_{i}\}_{i=1}^M,\{d_{i,j}\}_{i,j=1}^M)=  \sum_{i=1}^M \left(1 - \left(1+ e^{f_{\lambda}(\tilde{x}_{i})}\right)^{-L_{\pi}(i)}\right).\] 
\begin{remark}
The costs defined above are by no means exhaustive. We chose to define operational costs this way because they mimic the well known minimum latency objective in routing problems. 
For instance,  we could have used a \textcolor{black}{Poisson} failure model at each node instead of binomial or geometric as in Costs 1 and 2. Let us assume that the \textcolor{black}{Poisson} rate parameter $\mu(\tilde{x}_{\pi(i)})$ is the output of the estimation problem (say proportional to $p(\tilde{x}_{\pi(i)})$). Then 
\begin{equation*}
P(k \textrm{ failures occur in time } L_{\pi}(\pi(i))) = \frac{(\mu(\tilde{x}_{\pi(i)}) L_{\pi}(\pi(i)))^{k}e^{-\mu(\tilde{x}_{\pi(i)}) L_{\pi}(\pi(i))}}{k!}.
\end{equation*}
From this we can get the probability that at least one failure occurs in time interval $[0,L_{\pi}(\pi(i))]$ at node $\pi(i)$. Now we can define the operational cost to be the sum of these probabilities which depend on the routing and proceed in the same way as Cost 2. That is, we can minimize this cost to get the optimal routing $\pi^{*}$.
\end{remark}
\begin{remark} 
The operational cost must depend on graph properties like latency. We would not like to minimize an objective of the form $\sum_{i=1}^{M}\frac{1}{p(\tilde{x}_{\pi(i)})}$ (or any other function of just $p(\tilde{x}_{\pi(i)})$, the output of the estimation problem) as this does not lead to an operational cost in the true sense. This operational cost does not make use of latency information or other graph properties related to routing unless $p(\tilde{x}_{\pi(i)})$ implicitly depends on them (which is not the case here). 
\end{remark}
}
 \textcolor{black}{
 Now that the major steps for both formulations have been defined, we will discuss methods for optimizing the objectives. 
}

\section{Optimization}\label{sec:minlps}
 We start by formulating mixed-integer linear programs (MILP's) for the TRP subproblem.

\subsection{Mixed-integer optimization for Cost 1}
For either the sequential or simultaneous processes, we need the solution of the subproblem:
$\pi^* \in \argmin_{\pi\in\Pi} \textrm{OpCost}(\pi,f_{\lambda}^*,\{\tilde{x}_{i}\}_{i=1}^M,\{d_{i,j}\}_{i,j=1}^M)$,  or equivalently,
\begin{eqnarray}
\pi^{*}&\in& \argmin_{\pi\in\Pi} \sum_{i=2}^M p(\tilde{x}_{\pi(i)})
\sum_{k=1}^M d_{\pi(k) \pi(k+1)}  \one_{[k < i]} + p(\tilde{x}_{\pi(1)})
\sum_{k=1}^M d_{\pi(k) \pi(k+1)}.
\label{mytrp}
\end{eqnarray}
\textcolor{black}{
Let us compare this to the standard traveling repairman problem (TRP) problem \citep[see][]{blum94}:
\begin{equation}
\pi^*\in\argmin_{\pi \in \Pi}  \sum_{k=1}^M   d_{\pi(k) \pi(k+1)}(M+1-k).
\label{eqn:standard_trp}
\end{equation}
The standard TRP objective (\ref{eqn:standard_trp}) is a special case of the weighted TRP (\ref{mytrp}) when $\forall i=1,...,M, \; p(\tilde{x}_i)=p$:
\begin{eqnarray*}
\lefteqn{\sum_{i=2}^M p(\tilde{x}_{\pi(i)}) \sum_{k=1}^M d_{\pi(k) \pi(k+1)}  \one_{[k < i]} + p(\tilde{x}_{\pi(1)}) \sum_{k=1}^M d_{\pi(k) \pi(k+1)}}\\
&=& p \sum_{i=2}^{M} \sum_{k=1}^{M} d_{\pi(k) \pi(k+1)}\one_{[k < i]} + p\sum_{k=1}^{M}d_{\pi(k) \pi(k+1)} \\
&=& p \sum_{i=2}^{M} \sum_{k=1}^{M} d_{\pi(k) \pi(k+1)}\one_{[k < i]} + p\sum_{k=1}^{M}d_{\pi(k) \pi(k+1)} \one_{[k < M+1]}\\
&=& p  \sum_{k=1}^{M} d_{\pi(k) \pi(k+1)} \sum_{i=2}^{M+1} \one_{[k < i]}
=   p\sum_{k=1}^M d_{\pi(k) \pi(k+1)} (M+1-k).
\end{eqnarray*}
}
\textcolor{black}{
The TRP is different from the traveling salesman problem (TSP); the goal of the traveling salesman problem is to minimize the total traversal time (in this case, this is the same as the distance traveled) needed to visit all nodes once, whereas the goal of the traveling repairman problem is to minimize the sum of the waiting times to visit each node. Both the TSP and the TRP are known to be NP-complete in the general case \citep{blum94}. 
Intuitively, a TRP route cost objective captures the total waiting cost of a service system from the customer's (the node's) point of view. For example, consider a truck carrying prioritized items to be delivered to customers. At each customer's stop, that customer's item is removed from the truck. The goal of the TRP is to minimize the total waiting time of these customers.
} 

We start by extending an integer programming formulation of standard TRP \citep[][]{fischetti93} to include ``unequal flow values" so that we can solve (\ref{mytrp}) \textcolor{black}{
\citep[there are many other integer programming formulations in the literature as well, see for instance][]{lucena08}.
}
The weights $\{\bar{p}(\tilde{x}_{i})\}_i$ within the formulation below will be defined later. 
For interpretation, consider the sum of the probabilities $\sum_{i=1}^{M} \bar{p}(\tilde{x}_{i})$ as the total ``flow'' through a route. At the beginning of the tour, the repair crew has flow $\sum_{i=1}^{M} \bar{p}(\tilde{x}_{i})$. Along the tour, flow of the amount $\bar{p}(\tilde{x}_{i})$ is dropped when the repair crew visits node $\pi(i)$ at latency $L_{\pi}(\pi(i))$. In this way, the amount of flow during the tour is the sum of the probabilities $\bar{p}(\tilde{x}_{i})$ for nodes that the repair crew has not yet visited.  
We introduce two sets of variables $\{z_{i,j}\}_{i,j}$ and $\{y_{i,j}\}_{i,j}$ that together represent a route (instead of the $\pi$ notation). 
Let $z_{i,j}$ represent the flow on edge $(i,j)$ and let a binary variable $y_{i,j}$ represent whether there exists a flow on edge $(i,j)$. (There will only be a flow along the route, and there will not be a flow along edges that are not in the route.) The mixed-integer program is as follows:
\begin{eqnarray}
\min_{z,y}  \sum_{i=1}^M \sum_{j=1}^M   d_{i,j} z_{i,j} \textrm{\quad s.t.}\label{eqwtrp:obj}\\
\textrm{No flow from node $i$ to itself:}\;\; z_{i,i} = 0\;\; \forall i=1,...,M \label{eqwtrp:diag_ele1}\\
\textrm{No edge from node $i$ to itself:}\;\;y_{i,i} = 0 \;\;\forall i=1,...,M \label{eqwtrp:diag_ele2}\\
\textrm{Exactly one edge into each node:}\;\; \sum_{i=1}^M y_{i,j} = 1 \;\;\forall j=1,...,M \label{eqwtrp:colsum}\\
\textrm{Exactly one edge out from each node:}\;\;\sum_{j=1}^M y_{i,j} = 1 \;\;\forall i=1,...,M\label{eqwtrp:rowsum}\\
\textrm{Flow coming back to initial point at the end of loop:}\;\;\sum_{i=1}^M z_{i,1} = \bar{p}(\tilde{x}_1) \label{eqwtrp:lastflow}\\
\nonumber
\textrm{Change of flow after crossing node $k$:}\\
\sum_{i=1}^M z_{i,k} - \sum_{j=1}^M z_{k,j} = \left\{
\begin{array}{ll}
\bar{p}(\tilde{x}_1) - \sum_{i=1}^M \bar{p}(\tilde{x}_i) & k=1 \\
\bar{p}(\tilde{x}_k) & k=2,...,M
\end{array}
\right.\label{eqwtrp:interflow} \\
\textrm{Connects flows $z$ to indicators of edge $y$:}\;\;\;
z_{i,j} \leq r_{i,j} y_{i,j}\label{eqwtrp:x_ub} \\
\textrm{where  } r_{i,j} = \left\{
\begin{array}{ll}
\bar{p}(\tilde{x}_1) & j=1 \\
\sum_{i=1}^M \bar{p}(\tilde{x}_i) & i=1\\
\sum_{i=2}^M \bar{p}(\tilde{x}_i) & \mbox{otherwise.}\\
\end{array}
\right. \nonumber
\end{eqnarray}
Constraints (\ref{eqwtrp:diag_ele1}) and (\ref{eqwtrp:diag_ele2}) restrict self-loops from forming. Constraints (\ref{eqwtrp:colsum}) and (\ref{eqwtrp:rowsum}) ensure that every node should have exactly one edge coming in and one going out. Constraint (\ref{eqwtrp:lastflow}) represents the flow on the last edge coming back to the starting node. Constraint (\ref{eqwtrp:interflow}) quantifies the flow change after traversing a node $k$. Constraint (\ref{eqwtrp:x_ub}) represents an upper bound on $z_{i,j}$ relating it to the corresponding binary variable $y_{i,j}$. We can define the weights $\bar{p}(\tilde{x}_{i})$, for example, for Cost 1, to be equal to the estimated failure probabilities $1/(1+e^{-\lambda \cdot \tilde{x}_i })$.

\subsection{Mixed integer optimization for Cost 2}
\label{subsec:jointminimize_model2}
\textcolor{black}{
Here we reason about the choice for changing the cost per node in (\ref{model2costLi}) to resemble (\ref{alternative}). 
}
Starting with the sum (\ref{Cost2}) \textcolor{black}{ over node costs  (\ref{model2costLi})}, 
we apply the $\log$ function to the second term of the cost of each node \textcolor{black}{(\ref{model2costLi})} 
to get a new cost 
$
\left(1 - \log\left(1+ e^{f_{\lambda}(\tilde{x}_{\pi(i)})}\right)^{-L_{\pi}(\pi(i))}\right),
$
and the new minimization problem is: 
\begin{eqnarray*}
\min_{\pi}& \sum_{i=1}^{M}& \left( 1 - \log \left(1+ e^{f_{\lambda}(\tilde{x}_{\pi(i)})}\right)^{-L_{\pi}(\pi(i))} \right)\\
 &=& -\max_{\pi} \left( \sum_{i=1}^{M} \log \left(1+ e^{f_{\lambda}(\tilde{x}_{\pi(i)})}\right)^{-L_{\pi}(\pi(i))} - \textcolor{black}{\textrm{const}} \right)\\
 &=&  \min_{\pi} \left[\sum_{i=1}^{M} L_{\pi}(\pi(i)) \log \left(1+ e^{f_{\lambda}(\tilde{x}_{\pi(i)})}\right)\right] + \textcolor{black}{\textrm{const}},
\end{eqnarray*}
where the first term is the sum over nodes of the expression (\ref{alternative}).
This failure cost term is now a weighted sum of latencies where the weights are of the form $\log\left(1+ e^{f_{\lambda}(\tilde{x}_{\pi(i)})}\right)$. We can thus reuse the mixed integer program (\ref{eqwtrp:obj})-(\ref{eqwtrp:x_ub}) where the weights are redefined as $\bar{p}(\tilde{x}_i):= \log\left(1+e^{\lambda \cdot \tilde{x}_i }\right)$.

\textcolor{black}{
Our choices for the cost and failure models above allow us to use
a weighted version of the intuitive minimum latency or TRP problem for routing. In particular, the log transformation of individual terms in the original version of Cost 2, (\ref{Cost2}), precisely serves this purpose. 
}
\textcolor{black}{
In general, depending on the way we define the operational cost and the failure model, they may not necessarily map back to popular routing problems like the TRP as we have here. Nonetheless, there are many valid approaches beyond what we pursue this in this paper.
}

\textcolor{black}{
Now that the TRP subproblem has been completely defined for both Cost 1 and Cost 2, we will discuss first how to solve the subproblem alone, which is Step 2 of the sequential process. Then we will discuss the solvers for the simultaneous process.
}

\subsection{Solving the weighted TRP subproblem}
\textcolor{black}{
A generic MILP solver like CPLEX\footnote{IBM ILOG CPLEX Optimization Studio v12.2.0.2 2010} or Gurobi\footnote{Gurobi Optimizer v3.0, Gurobi Optimization, Inc. 2010} can produce an exact solution using branch-and-bound or other related exact methods. We use Gurobi. 
 The weighted TRP problem is NP-hard (can be shown by a reduction to the \textcolor{black}{Hamiltonian} cycle problem) and hence most likely not solvable by polynomial-time algorithms. The standard unweighted (all weights equal) TRP can be encoded by different mixed-integer programming formulations \citep[see][]{fischetti93,eijl95,lucena08} each with different performance guarantees (e.g., solving 15-60 nodes), which could be adapted for our purpose. 
 There are also techniques for producing constant factor approximate solutions to the unweighted TRP \citep{goemans98,blum94,arora00,archer08,archer10}, which could run faster than the MILP solvers for large problems. 
 If the weights $\{w_i\}_i$ are integers, we can adapt these faster techniques for the standard problem to the weighed TRP problem by replicating each node $w_{i}$ times. If the weights are rational, as is the case in (\ref{eqn:learning1}) and (\ref{eqn:learning2}), we can use rounding and discretization in order to apply the faster solution techniques for solving the standard TRP. 
}

\subsection{Solving Mixed-integer nonlinear programs (MINLPs)}
\label{subsec:MINLPsolve}
\textcolor{black}{
For the simultaneous process, the inputs to the program are training data $\{x_i,y_i\}^{m}_{i=1}$, unlabeled nodes $\{\tilde{x}_i\}_{i=1}^{M}$ the distances between them $\{d_{i,j}\}_{i,j=1}^M$ and constants $C_{1}$ and $C_{2}$. 
}
The full simultaneous process formulation using Cost 1 is:
\begin{eqnarray}
\min_{\lambda} \left(\sum_{i=1}^m \ln\left(1+e^{-y_if_{\lambda}(x_i)}\right)+C_{2}\|\lambda\|_{2}^{2} 
+C_{1}\min_{\{z_{i,j},y_{i,j}\}}\sum_{i=1}^M \sum_{j=1}^M   d_{i,j}z_{i,j} \right) \label{eqn:learning1}\\
\nonumber \textrm{subject to constraints (\ref{eqwtrp:diag_ele1}) to (\ref{eqwtrp:x_ub}), where } \bar{p}(\tilde{x}_{i}) = \frac{1}{1+e^{-\lambda \cdot \tilde{x}_i }}.
\end{eqnarray} 
The full formulation using the modified version of Cost 2 is:
\begin{eqnarray}
\min_{\lambda} \left(\sum_{i=1}^m \ln\left(1+e^{-y_if_{\lambda}(x_i)}\right) +C_{2}\|\lambda\|_{2}^{2} 
+C_{1}\min_{\{z_{i,j},y_{i,j}\}}\sum_{i=1}^M \sum_{j=1}^M   d_{i,j}z_{i,j} \right) \label{eqn:learning2}\\
\nonumber \textrm{subject to constraints (\ref{eqwtrp:diag_ele1}) to (\ref{eqwtrp:x_ub}) hold, where } \bar{p}(\tilde{x}_{i}) =\log\left(1+e^{\lambda \cdot \tilde{x}_i }\right).
\end{eqnarray} 
\textcolor{black}{
If we have an algorithm for solving (\ref{eqn:learning1}), then the same scheme can be used to solve (\ref{eqn:learning2}). 
}
\textcolor{black}{
There are multiple ways of solving (or approximately solving) a mixed integer nonlinear optimization problem of the form (\ref{eqn:learning1}) or (\ref{eqn:learning2}). 
}
We consider three methods in this paper for solving (\ref{eqn:learning1}) and (\ref{eqn:learning2}).
\begin{itemize}
\item Generic mixed integer non-linear programming (MINLP) solver (Bonmin). 
\item Nelder-Mead (NM) which is a iterative scheme over the $\lambda$ parameter space, solving a weighted TRP subproblem in each iteration. 
\item Alternating Minimization (AM) which alternatively minimizes over $\lambda$ and $\pi$ optimization variables.
\end{itemize}
\textcolor{black}{
 \subsubsection*{Method 1: MINLP Solver}
 For our experiments we directly use a MINLP solver called Bonmin \citep{bonmin}. These types of solvers typically use general MILP solving techniques like branch and bound or dynamic programming interleaved with continuous optimization. Since the general MILP solving techniques, as discussed, can take exponential time when applied directly to our formulations, the MINLP solvers which use them can in turn, be inefficient if the graph is moderate to large in size. However, when the graph is small, for instance when we want to schedule a tour over only a few nodes, the MINLP solver can generally compute a solution to the problems (\ref{eqn:learning1}) or (\ref{eqn:learning2}) in a manageable period of time.
 }
 \textcolor{black}{
 \subsubsection*{Method 2: Nelder-Mead in $\lambda$-space (NM)}
 The Nelder-Mead minimization algorithm requires only function evaluations \citep{neldermead}. The ML\&TRP can be viewed as a minimization in the space of all $\lambda$ vectors; since we have solvers for the weighted TRP subproblem, we are able to evaluate the ML\&TRP objective for a given value of $\lambda$. In our experiments we use the MILP solver (Gurobi) for the subproblem. Note that the ML\&TRP objective can have non-differentiable kinks arising from discontinuities in the failure cost term; a method that relies on the gradient or Hessian of the objective function might get stuck in narrow local minima, whereas methods that use only function evaluations may not have this problem. The generic Nelder-Mead scheme can have disadvantages with respect to performance \citep{rios09}, in which case, other schemes like Multilevel Coordinated Search (MCS) \citep{neumaier98} can be used in place of Nelder-Mead. Note that since the objective is non-convex, all solutions obtained by NM are only guaranteed to be locally optimal.
 }

\subsubsection*{Method 3: Alternating minimization in $\lambda$-$\pi$ space (AM)}
Our alternating minimization scheme also operates in the $\lambda$-$\pi$ space
as follows. Define the objective Obj as a function of $\lambda$ and $\pi$: 
\begin{equation*}
\textrm{Obj}(\lambda,\pi) =  
\sum_{i=1}^m \ln \left(1+e^{-y_if_{\lambda}(x_i)}\right) + C_{2}\|\lambda\|_{2}^{2}
+C_{1} \textrm{OpCost}\left(\pi,f_{\lambda},\{\tilde{x}_{i}\}_{i=1}^M,\{d_{i,j}\}_{i,j=1}^M\right).
\end{equation*}
Starting from an initial vector $\lambda_{0}$, Obj is minimized alternately with respect to $\lambda$ and then with respect to $\pi$, as shown in Algorithm \ref{figure:am}. 
The second step, solving for $\pi$, is the same as solving the TRP subproblem, and we again use the MILP solver for this. Conditions for convergence and correctness for such iterative schemes are given by \citet{csiszar84}; again, it is not possible to guarantee globally optimal solutions using this method.

\begin{algorithm}[t]
\label{algAM}
\begin{algorithmic}
\STATE 
\STATE \textbf{Inputs:} $\{x_{i},y_{i}\}_{1}^{m}, \{\tilde{x}_{i}\}_{1}^{M}, \{d_{ij}\}_{ij}, C_{1}, C_{2}, T$ and initial vector $\lambda_{0}$.

\FOR{t=1:T}
        \STATE Compute $\pi_t \in \argmin_{\pi\in \Pi}  \textrm{Obj}(\lambda_{t-1},\pi)$.
        \STATE  Compute $\lambda_t \in \argmin_{\lambda\in \R^d}  \textrm{Obj}(\lambda,\pi_{t})$. 
\ENDFOR 
\STATE \textbf{Output:} $\pi_T$.
\end{algorithmic}
\caption{AM: Alternating minimization algorithm\label{figure:am}}
\end{algorithm}


\subsection{Illustrative Experiment}\label{subsec:illus}
We will use the ML\&TRP to show the fundamental property motivating the MLOC framework: that
a large change in the probability model does not necessarily lead to a large change in overall prediction accuracy, but may lead to very different solutions. 

The training set was chosen uniformly at random from a distribution that is uniform over two triangles pointing end to end. We used six unlabeled points as the nodes. See Figure \ref{fig:example6node_feature}. In addition a level set, colored black, is also plotted. It is the estimated level set for $P(y=1|x)= 0.5$ learned from $\ell_{2}$-regularized logistic regression. A second level set, colored red, also drawn at probability estimate $0.5$, is learned from the simultaneous process, with failure cost modeled according to Cost 1. Now, node $6$ (triangle with label ``$\tilde{x}_{6}$'') lies in a low density region of feature space, so its probability cannot be well estimated. For the sequential formulation, node $6$ was assigned $p(\tilde{x}_{6}) = 0.5$ and the optimal route obtained by solving the weighted TRP problem is 1-2-3-6-4-5-1, shown in Figure \ref{fig:physical_space_2step}. \textcolor{black}{The node represented by $\tilde{x}_{1}$ is chosen to be the starting point}. For the simultaneous process, node $6$ has been assigned a new probability value $p(\tilde{x}_{6}) = 0.29$. This change is possible because node 6's probability estimate can vary quite a lot without changing the probability estimates of others. This changes the route to 1-2-3-4-5-6-1 as shown in Figure \ref{fig:physical_space_joint}.

In the simultaneous process, we chose $C_{1}$ large enough so that the tour route visits $4$ and $5$ before $6$. This results in a $\sim 9$\% decrease in the failure cost (Cost 1), with a $\sim 3$\% change in the training error (logistic loss). In particular, for the sequential process, Cost 1 is 4.7 units and the training error is 15.7 units; for the simultaneous process, Cost 1 is 4.25 units and the learning error is 16.2 units ($C_{1} = 5\times10^{-4}$). This is an illustration of the core of MLOC: both predictive models are good, and a range of operational costs and decisions exist between them.

\begin{figure}
\centering
\subfigure[]{
        \begin{overpic}[width=0.35\textwidth]{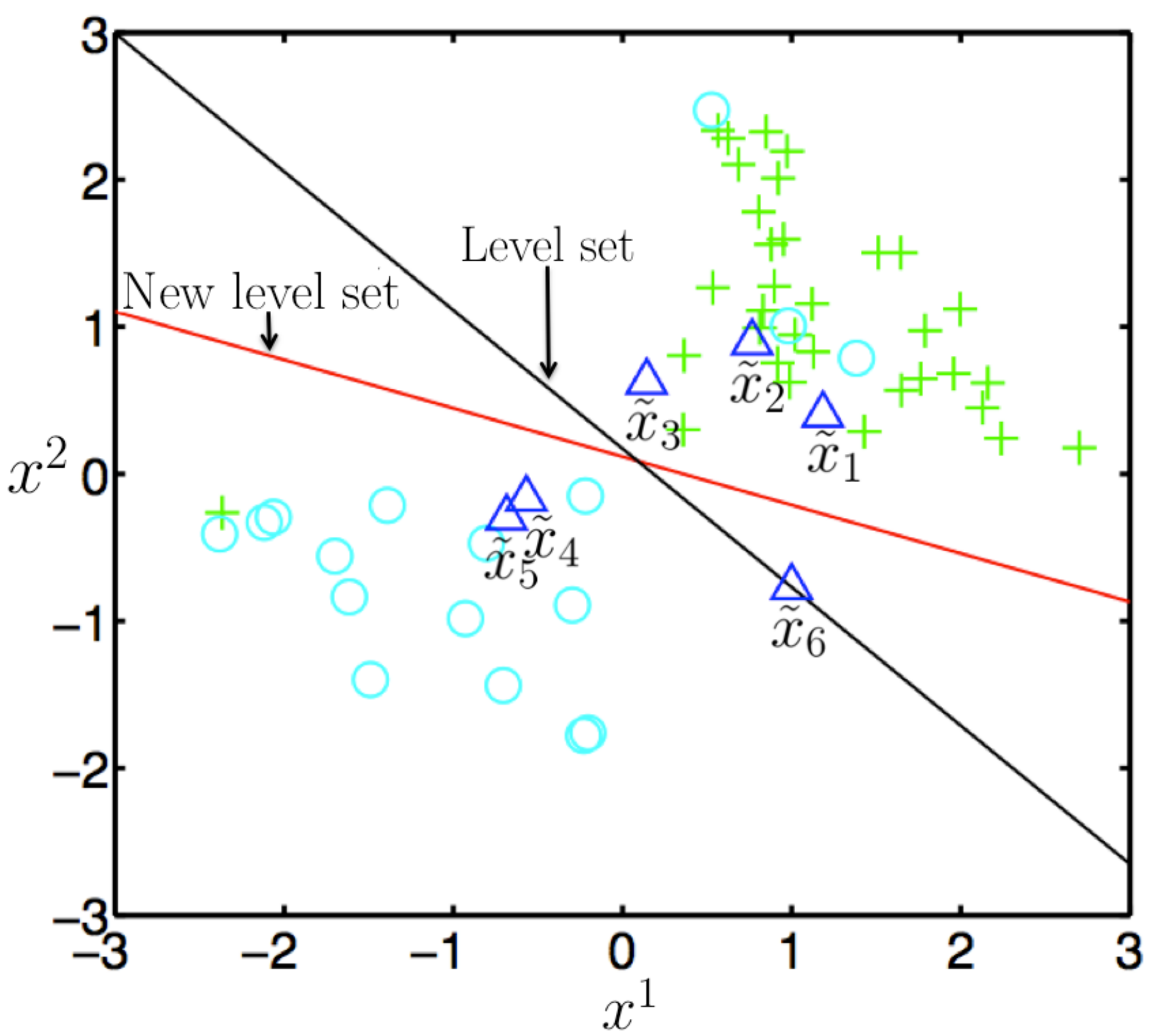}
        \end{overpic}
     \label{fig:example6node_feature}
     }
\subfigure[]{
        \includegraphics[width=0.6\textwidth]{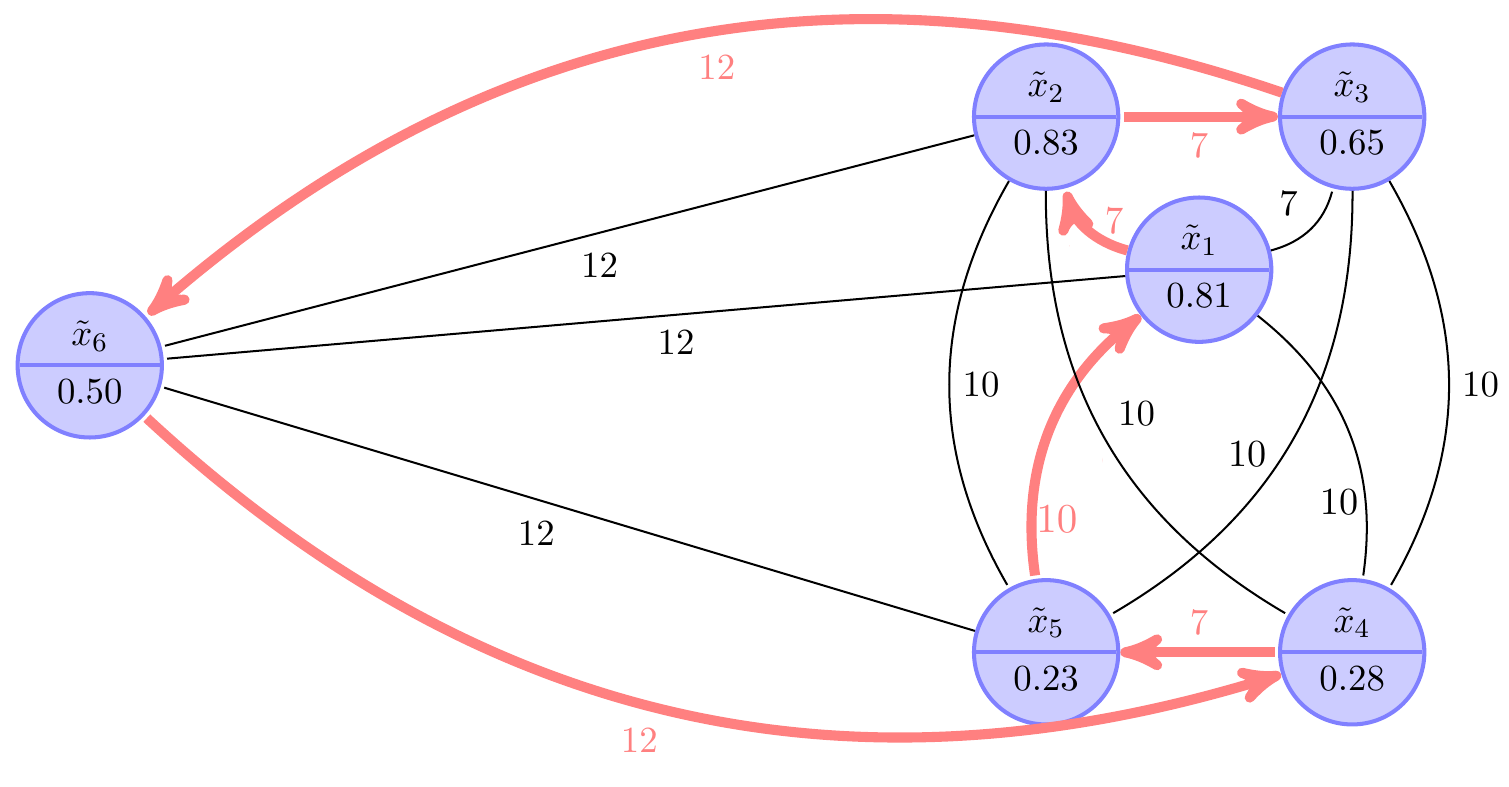}
        \label{fig:physical_space_2step}
     }
\caption{Left: $x^{1}$ and $x^{2}$ represent the first and second coordinates respectively of the 2D feature space. 
The triangles represent the unlabeled data $\{\tilde{x}_{i}\}_{i=1}^{6}$.
Right: The numbers in the nodes indicate their probability of failure, and the numbers on the edges indicate distances. The optimal route 1-2-3-6-4-5-1 as determined by the sequential formulation is highlighted.
}
\end{figure}

\begin{figure}
\centering
\includegraphics[width=0.6\textwidth]{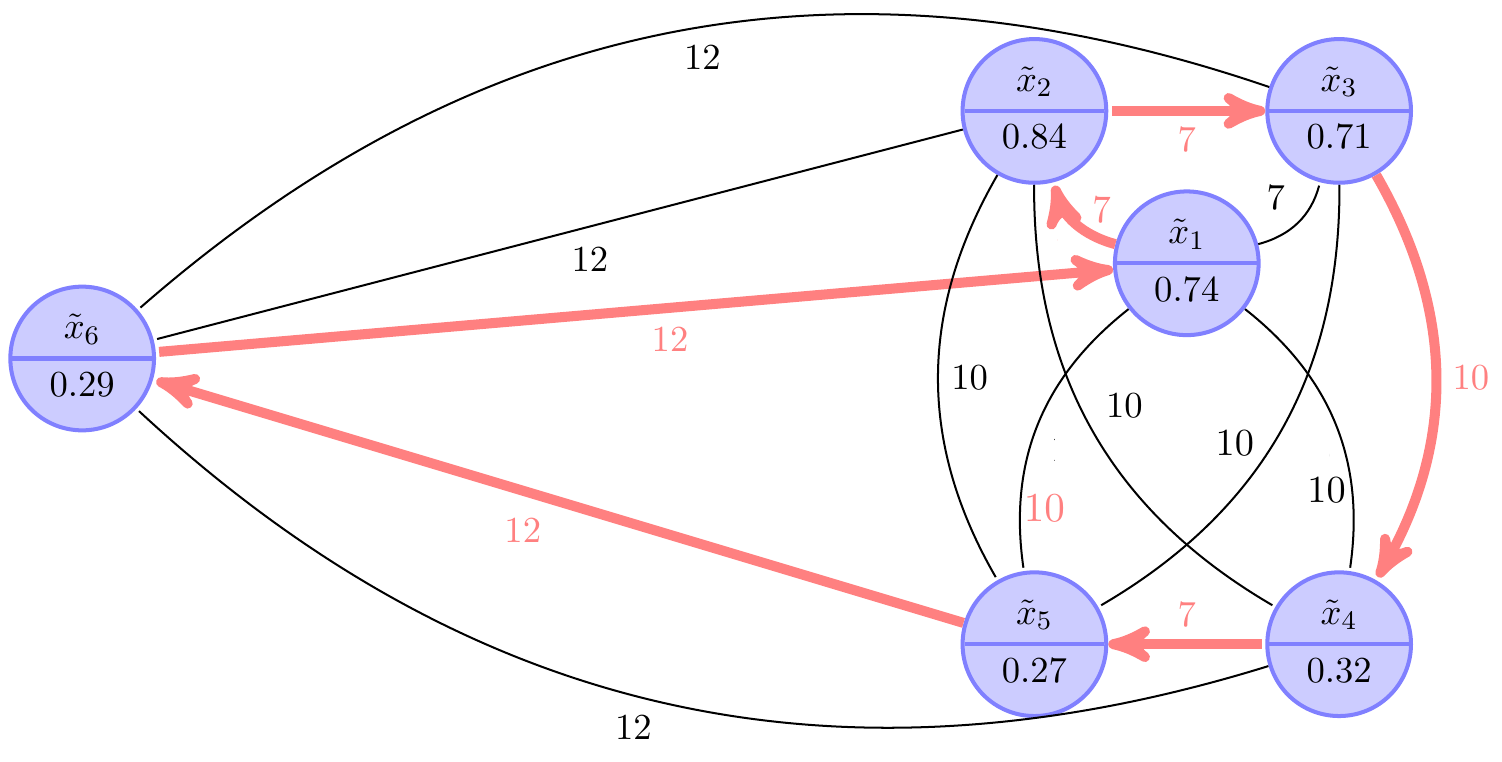}
\caption{The optimal route 1-2-3-4-5-6-1 determined by the simultaneous process is highlighted.
\label{fig:physical_space_joint}}
\end{figure}

\section{ML\&TRP on the NYC power grid}\label{sec:grid}
\textcolor{black}{
We now show how the MLOC framework might be used to assist companies like Con Edison, which is NYC's power utility company. We pursue three sets of experiments. The first experiment demonstrates the use of the simultaneous process when given a specific routing problem. This shows how a practitioner would use the simultaneous process in practice. In the second experiment, we randomize over the training sample and routing problems.  This experiment shows that the simultaneous process can find models that are equally predictive or better than the sequential method when operational costs are included. In the third experiment, we look at scaling issues. 
}

\textcolor{black}{
In all these experiments, we are predicting the probability of failure over the course of a year. While using the predicted failure probabilities in the routing problem, we will assume that these are probabilities of failures in an arbitrary unit interval of time. In particular, they can be the probability of failures over an hour, a day etc. We make the approximation that the probabilities at finer time scales (required for the routing problem) are proportional to the probabilities at coarser time scales for the purpose of our experiments.
}

\subsection{The dataset}
The dataset we use is described by  \citet{RudinETAL10}, which was developed in order to assist Con Edison with its maintenance and repair programs on the secondary electrical distribution network in NYC; specifically, it was designed for the purpose of predicting manhole fires and explosions. We chose to use all manholes from the Bronx ($\sim$23K manholes). Each manhole is represented by (4-dimensional) features that encode the number and type of electrical cables entering the manhole and the number and type of past events involving the manhole. The event features encode how often in the past the manhole was the source of partial outages, full outages and/or underground burnouts. The training features encode events prior to 2008, and the training labels are 1 if the manhole was the source of a serious event (fire, explosion, smoke) during 2008. The prediction task is to predict events in 2009. The test set (for evaluating the performance of the predictive model) consists of features derived from the time period before 2009, and labels from 2009. In our experiments, for both training and test we had a large sample (23,217 instances). There were 211 and 132 failure instances in the test and training data respectively.

\subsection{Performance of the simultaneous process for a seven node decision problem}
\textcolor{black}{In this experiment, the operational task is to design a route for a repair crew that is equipped to fix seven relatively more vulnerable manholes in 2009. The distances between the nodes were obtained from Google Maps, by querying the driving distance between each pair of nodes. Note that we do not want `flying' distance between two coordinates as this can be very different from the actual driving distance, especially in New York City.
}

\textcolor{black}{
The limited resources for inspection and repair of manholes should generally be designated to the most vulnerable manholes. With uncertainty in many of the probability estimates, if we are not careful, it is possible that most of these resources will be spent in dealing with outliers whose probabilities are overestimated. The simultaneous process will generally prevent this from happening if we choose $C_1$ to have a sufficiently large positive value.
}

\textcolor{black}{
Manhole failures are rare events. This means there are many more negative labels than positive labels. Using a logistic model gives probability estimates which are low overall, so the misclassification error is almost always the size of the whole positive class. Because of this, we evaluate the quality of the predictions from $f_{\lambda^{*}}$ using the area under the ROC curve (AUC), for both training and test. AUC is a measure of ranking quality; it is sensitive to the rank-ordering of the nodes in terms of their probability to fail, and it is not as sensitive to changes in the values of these probabilities. This means that as the parameter $C_1$ increases, the estimated probability values will tend to decrease, and thus the failure cost will decrease.
}

\textcolor{black}{For the experiment, a specific decision problem was sampled and fixed a priori,  involving repairs on a handful of relatively more vulnerable manholes in the Bronx.}
We solved (\ref{eqn:learning1}) and (\ref{eqn:learning2}) for a range of values for the regularization parameter $C_1$, for both costs and all three methods, with the goal of seeing whether for the same level of  estimation performance, we can get a range in the cost of failures. In particular, we 
\textcolor{black}{wanted} to know if we \textcolor{black}{could} see a substantial reduction in the cost.
We varied $C_1$ so that the variation in the training error term across the methods was small, about 2\% away from the solution of the sequential process ($C_1=0$), see Figure \ref{fig:term1_loss_7nd_m1m2}. For that range, the test AUC values for the simultaneous process were all within 1\% of each other; this is true for both Cost 1 and Cost 2, for each of the AM, NM, and MINLP solvers, see Figures \ref{fig:plot_auc_7nd_m1} and \ref{fig:plot_auc_7nd_m2}. So, changing $C_{1}$ did not dramatically impact the prediction quality as measured by the AUC. On the other hand, the failure costs varied widely over the different methods and settings of $C_1$, as a result of the change in the probability estimates, as shown in Figure \ref{fig:traversal_cost_plot_7nd_m1m2}.
As $C_{1}$ was increased from $0.05$ to $0.5$, Cost 1 went from $27.5$ units to $3.2$ units, which is over eight times smaller. This means that with a 1-2\% variation in the predictive model's AUC, the operational cost can decrease a lot, yielding a completely different possible route for inspection and/or repair work. The reason for an order of magnitude change in the failure cost is because the probability estimates vary by an order of magnitude due to uncertainty at the nodes. This uncertainty in costs is what the MLOC allows us to uncover.

\begin{figure}
     \centering
     \subfigure[]{
     \begin{overpic}[width=.45\textwidth]{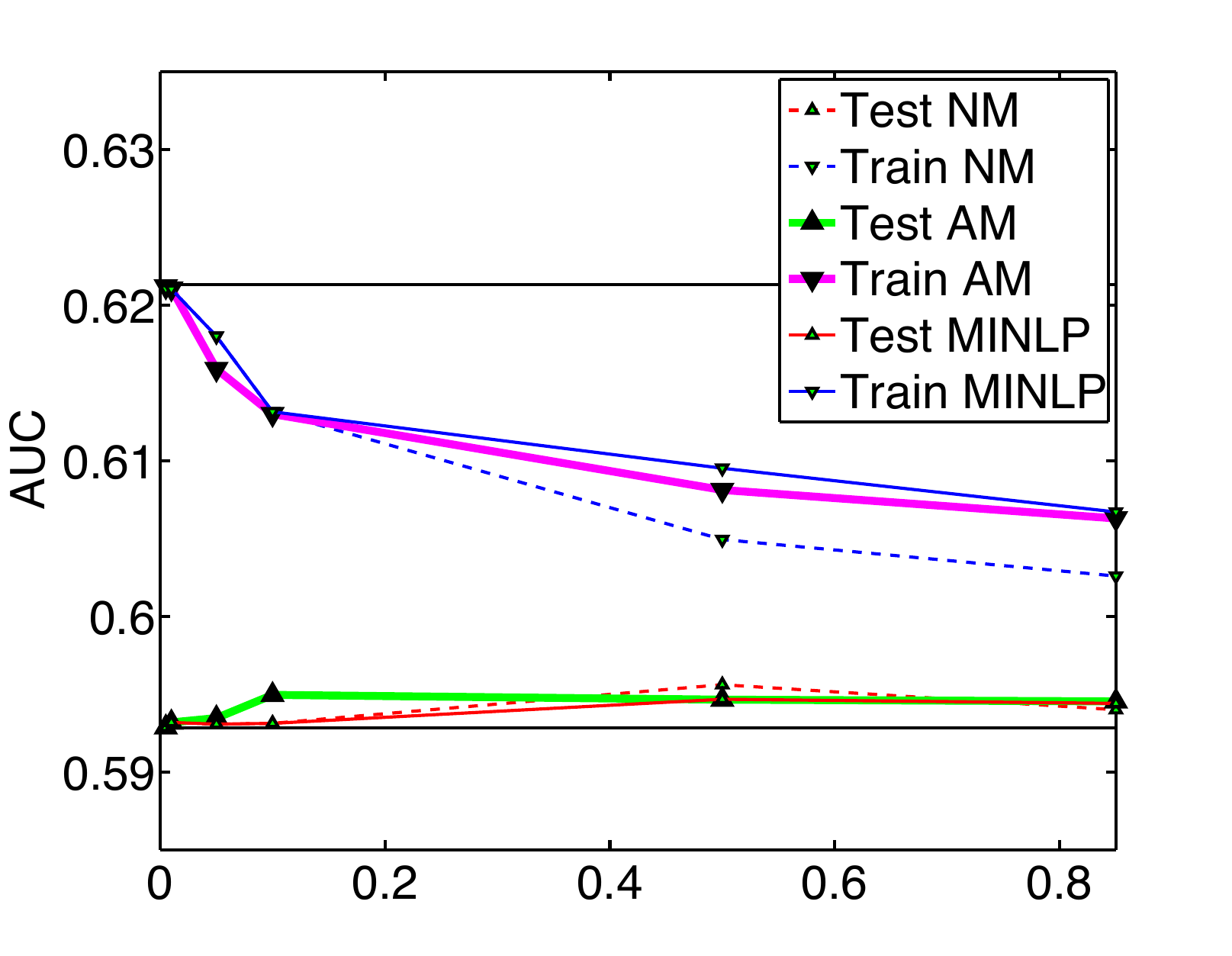}
          \put(80,0){\large $C_{1}$}
     \end{overpic}
     \label{fig:plot_auc_7nd_m1}
     }
     \subfigure[]{
      \begin{overpic}[width=.45\textwidth]{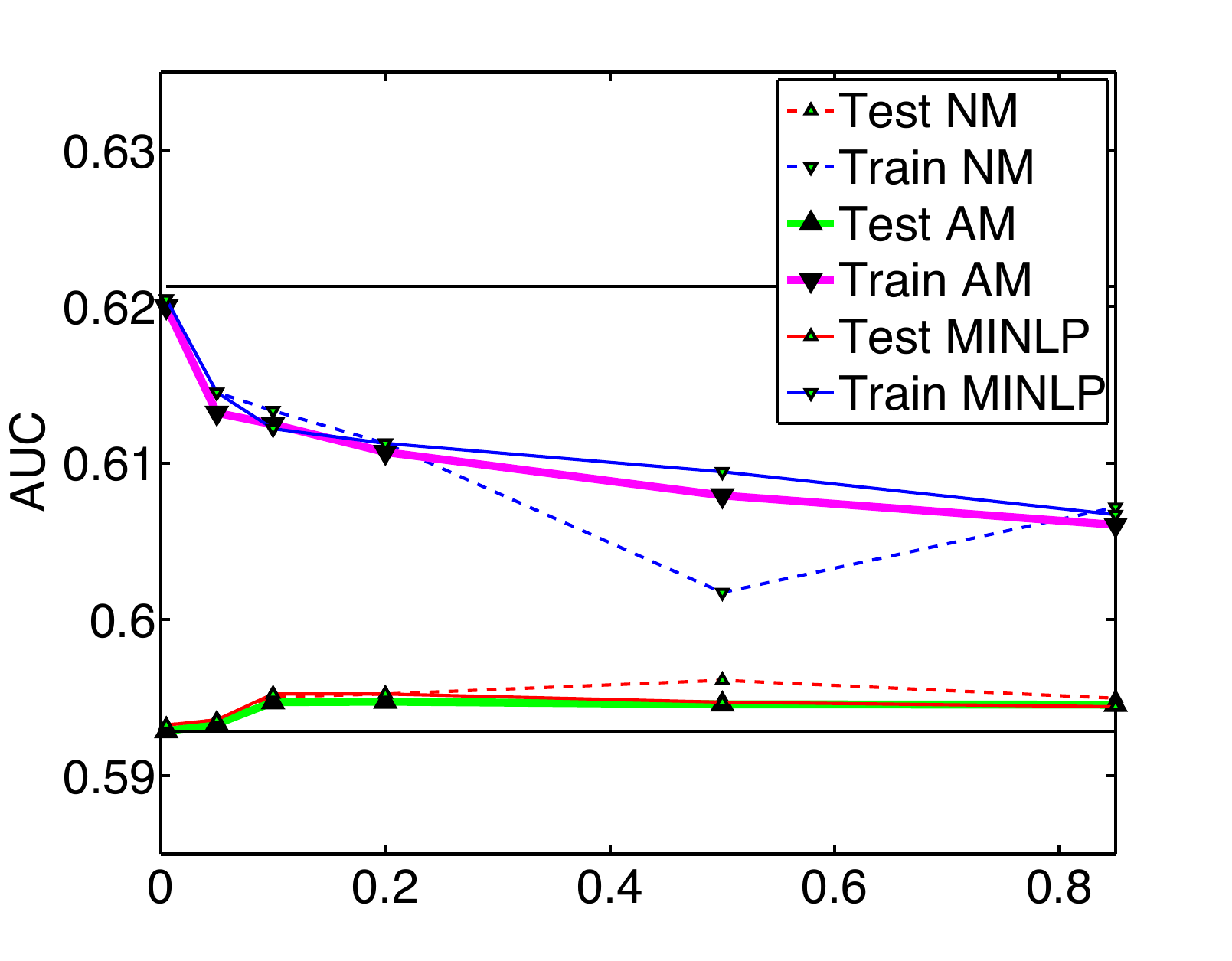}
           \put(80,0){\large $C_{1}$}
     \end{overpic}
     \label{fig:plot_auc_7nd_m2}
     }
     \caption{Left: The AUC values corresponding to models (parameterized by $C_{1}$) obtained from the simultaneous process using Cost 1 by NM and AM and MINLP techniques. The AUC values on the training data decrease slightly and the same values for test data increase marginally. The two horizontal lines represent the training and test AUC values obtained by $\ell_{2}$-penalized logistic regression are constant with respect to $C_{1}$. Right: Similar AUC values obtained from the simultaneous process, using Cost 2.
}
\end{figure}

\begin{figure}
     \centering
     \subfigure[]{
     \begin{overpic}[width=.45\textwidth]{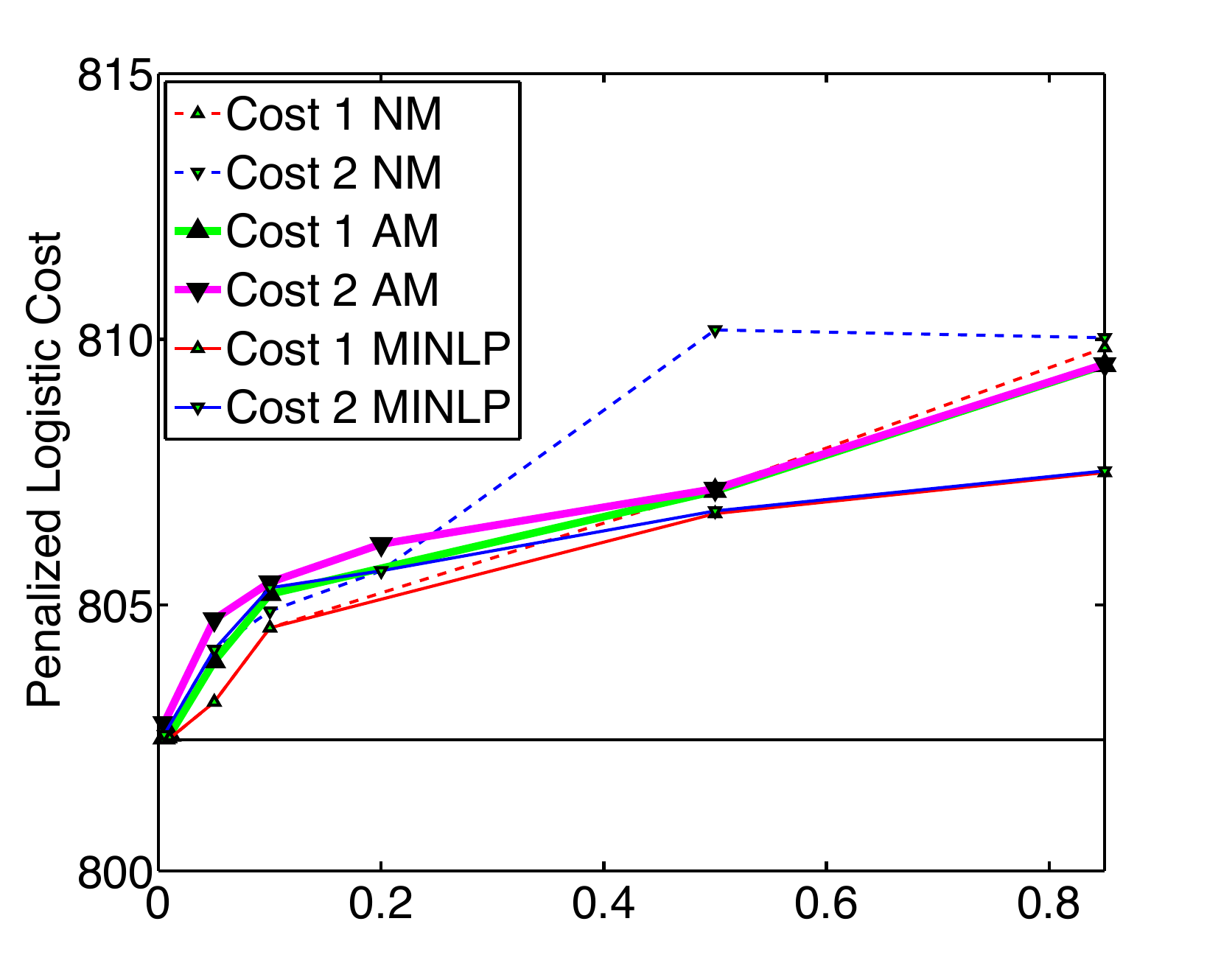}
     \put(80,0){\large $C_{1}$}
     \end{overpic}
     \label{fig:term1_loss_7nd_m1m2}
     }
        \subfigure[]{
     \begin{overpic}[width=.45\textwidth]{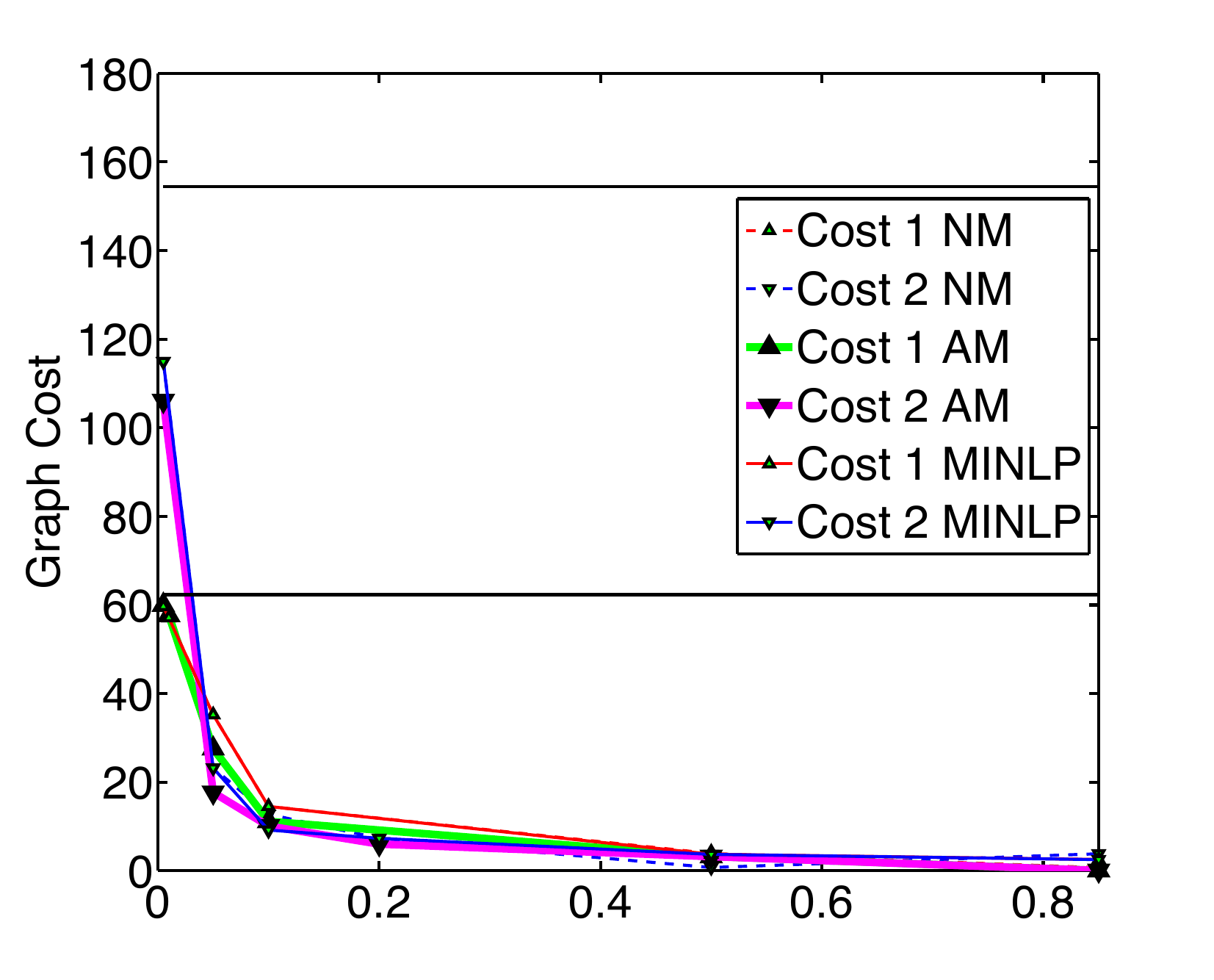}
          \put(80,0){\large $C_{1}$}
     \end{overpic}
     \label{fig:traversal_cost_plot_7nd_m1m2}
     }
     \caption{Left: The $\ell_{2}$-regularized logistic loss increases as a function of increasing $C_{1}$. The horizontal line represents the loss value from $\ell_{2}$-penalized logistic regression with no regularization ($C_{1} = 0$).    Right: The failure costs decrease as a function of the regularization parameter $C_{1}$. The horizontal lines in the figure represent the sequential formulation solution; the lower horizontal line is Cost 1 of the solution obtained by $\ell_{2}$-penalized logistic regression, and the upper line is Cost 2 of that solution. }
\end{figure}

In Figures \ref{fig:route_7nd_naive}-\ref{fig:route_7nd_optimal} we show the routes according to the different algorithms. We first provide the na\"ive route in Figure \ref{fig:route_7nd_naive}, which was obtained by estimating probabilities using $\ell_{2}$-penalized logistic regression, and then simply visiting nodes according to decreasing values of these probabilities. Figure \ref{fig:route_7nd_seq} shows the route provided by the sequential process. When the failure term starts influencing the optimal solution of the objective (\ref{eqn:learning1}) because of an increase in $C_{1}$, we get a new route, depicted in Figure \ref{fig:route_7nd_optimal}. In most applications relevant to this problem, we suspect that the solution used in practice is somewhere in between the na\"ive route and the sequential route, in that a human views the na\"ive solution and adjusts it by hand to be closer to the sequential route (without solving the TRP). For the application to electrical grid maintenance, the simultaneous process was able to find a substantially lower cost route  \textcolor{black}{than} the na\"ive or sequential process, with little (if any) change in the AUC prediction quality. This demonstration on data from the Bronx indicates that it is possible to better understand uncertainty in modeling. If engineers truly believe the costs will be lower, their belief, combined with the route we found, can be used to justify a much more cost-effective solution.

\begin{figure}
     \centering
     \subfigure[]{
     \includegraphics[width=.31\textwidth]{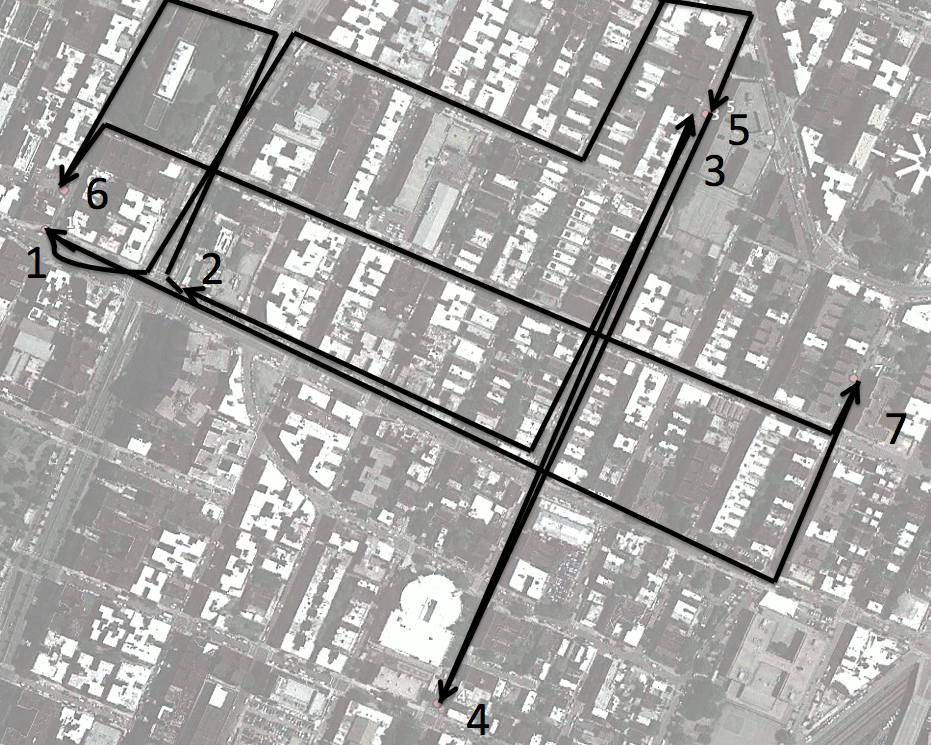}
     \label{fig:route_7nd_naive}
     }
     \subfigure[]{
     \includegraphics[width=.31\textwidth]{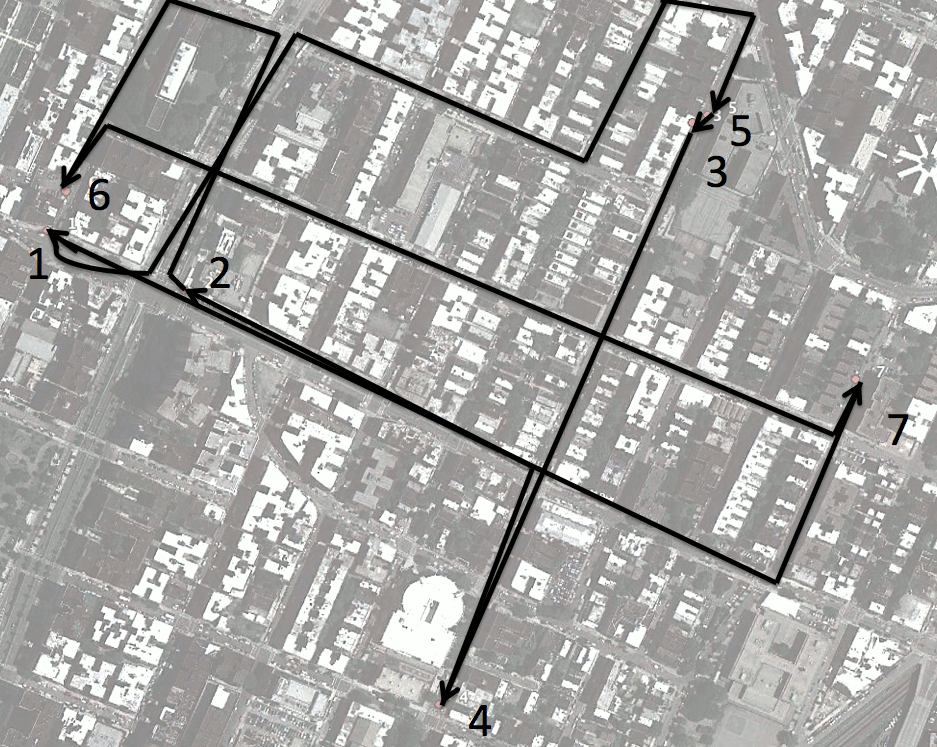}
     \label{fig:route_7nd_seq}
     }
     \subfigure[]{
     \includegraphics[width=.30\textwidth]{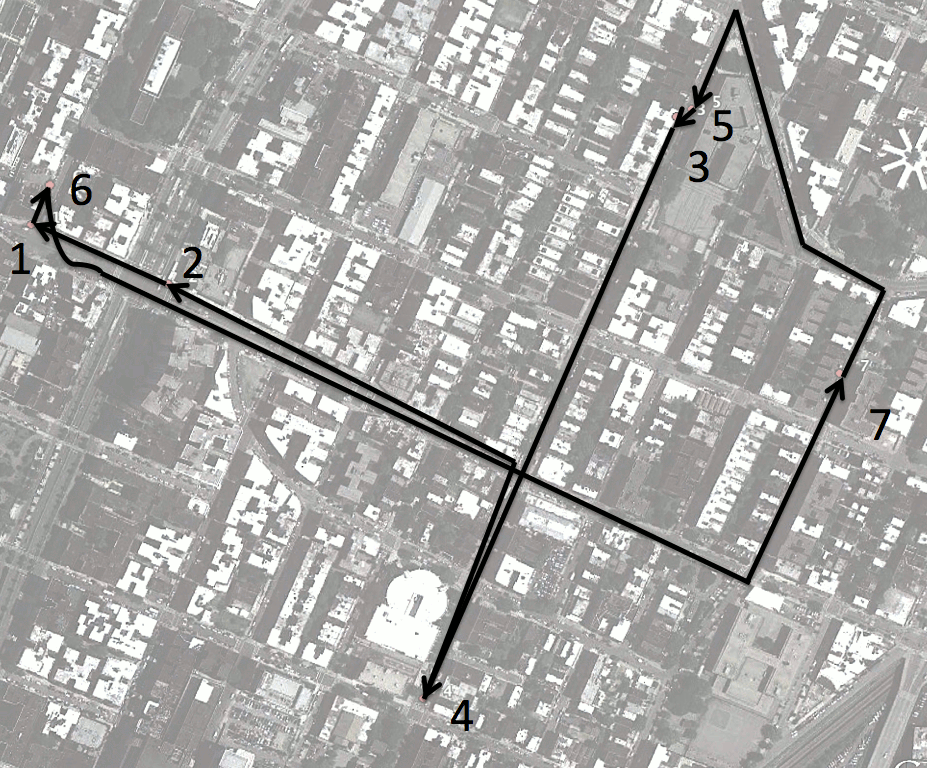}
     \label{fig:route_7nd_optimal}
     }
     \caption{Left: A na\"ive route: 1-5-4-3-2-6-7-1 obtained by sorting the probability estimates in decreasing order and visiting the corresponding nodes. Center: Sequential process route: 1-5-3-4-2-6-7-1. The simultaneous process also chooses this route when $C_1$ is small. Right: Route chosen by the simultaneous process when $C_1$ is larger: 1-6-7-5-3-4-2-1. Prediction performance is only slightly influenced by the route change, but the routing cost (Cost 1) decreases a lot.
      }
\end{figure}

\textcolor{black}{
\subsection{Performance of the simultaneous process across randomly generated decision problems}
In this experiment, we varied the size of the training data and characterized its effect on learning for both the sequential process and the simultaneous process. 
We expect to see that when the sample size is small, the operational cost regularization {can lead} to better performance for the simultaneous process  {for some $C_1$. That is, we are showing that some type of knowledge on the operational cost can be helpful in prediction.}  (When the sample size is large, the regularization term of the simultaneous process should not have much of an effect, and the sequential and simultaneous process models should perform similarly, which is unsurprisingly what we observe.)
}

\textcolor{black}{
To conduct the experiment, we considered training samples ranging from 10\% of the original training set size to 100\% of the original training set size. 
For each training set we generated, we then generated 100 seven
node decision problems (TRP problems) from a separate held out test set. Each decision problem was generated by randomly picking the nodes (whose labels are not known during training) and computing the distances between each pair of them. For each new training sample size and for each random decision problem, we solved the sequential process and the simultaneous process for both Cost 1 and Cost 2. 
}
\textcolor{black}{
In particular, this involved the following. 
\begin{itemize}
\item For the sequential process we performed a 5-fold cross validation to pick the coefficient for the $\ell_2$ regularization term. Once the optimal regularization constant was chosen, we computed the predicted probabilities of failure and solved the corresponding weighted TRP subproblem.
\item We solved the simultaneous process using the AM algorithm for 4 different $C_1$ values, and the one achieving the best test performance (on a separate held out test set) was reported.  {This encodes the notion that one of the $C_1$ values, namely the one which gives the best test performance, encodes the right prior knowledge.} In total, 8,000 mixed integer nonlinear programs were solved (4 $C_1$ value settings per decision problem (100) per training sample size (10) per decision cost type (Cost 1 and Cost 2)).
 \end{itemize} 
}

 \textcolor{black}{
Figure \ref{fig:test_performance_cost1} shows how the simultaneous process compares with respect to the sequential process in terms of AUC on a held out test set as the size of the training sample is varied for Cost 1. The x-axis shows different training sample sizes and the y-axis shows the difference between the AUC of a simultaneous process model (one for each training size and decision problem) and the AUC of the corresponding sequential process model, where 0 means that the AUC's for the two processes were identical. From the figures, we can infer the following:
\begin{itemize}
\item The test performance of the simultaneous process can often be better than that of the sequential process for smaller training sets. This is because at lower sample sizes, the simultaneous process gains an advantage from the prior knowledge about operational costs. 
\item At larger training set sizes, the logistic models from the simultaneous process and the sequential process performed similarly. Again this is not surprising, as the regularization becomes less influential as the training set size increases.
\end{itemize}
}

\begin{figure}
     \centering
     \includegraphics[width=.65\textwidth]{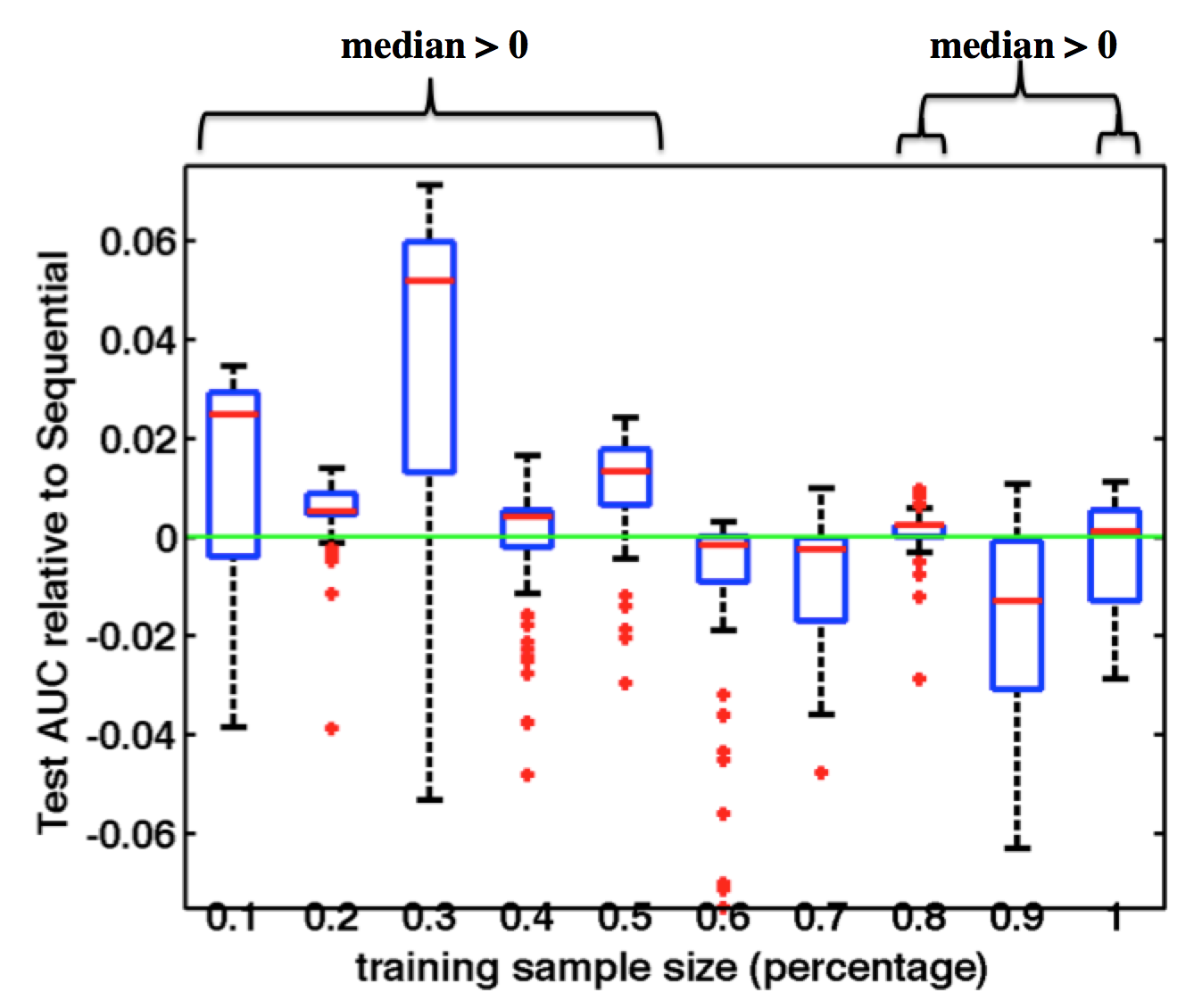}
     \caption{Performance of the two processes on randomly generated decision problems at various training sample sizes with Cost 1 as the routing cost. The evaluation is over a separate held out test set. The green solid line is the zero mark. For each size of the training sample on the x-axis (varying from 10\% to 100\% of the original training sample size), we solved the simultaneous process for 100 random seven node decision problems and the performances of the corresponding models relative to the sequential process models are plotted as a box-plot.     \label{fig:test_performance_cost1}
}
\end{figure}

\textcolor{black}{
At each training sample size, we tested two hypotheses using the (nonparametric) sign test, with significance level $\alpha=0.05$.
In the first test, the null hypothesis was that the median AUC performance of the two processes was the same versus the alternative that the median AUC performance of the simultaneous process is greater than the median AUC performance of the sequential process. For three of the larger training sample sizes (namely $.6, .7$ and $.9$ of the original), we could not reject the null as the corresponding p-values were greater than the significance level and for the remaining 7 training sample sizes, we could reject the null that the median performance of the two methods is the same. In the second test, the null hypothesis was that the median routing cost using the two processes was the same versus the alternative that the median routing cost of the simultaneous process is smaller than the median routing cost of the sequential process. Here, we were able to reject the null hypothesis for all 10 training sample sizes.
}

\textcolor{black}{We ran this experiment again with Cost 2 as the routing cost, and solved the same 100 decision problems for 4 different $C_1$ values for each of the 10 different training samples of different sizes. Figure \ref{fig:test_performance_cost2} summarizes the performance of these models. The inferences one can draw from this plot are similar to the previous case.}

\begin{figure}
     \centering
     \includegraphics[width=.65\textwidth]{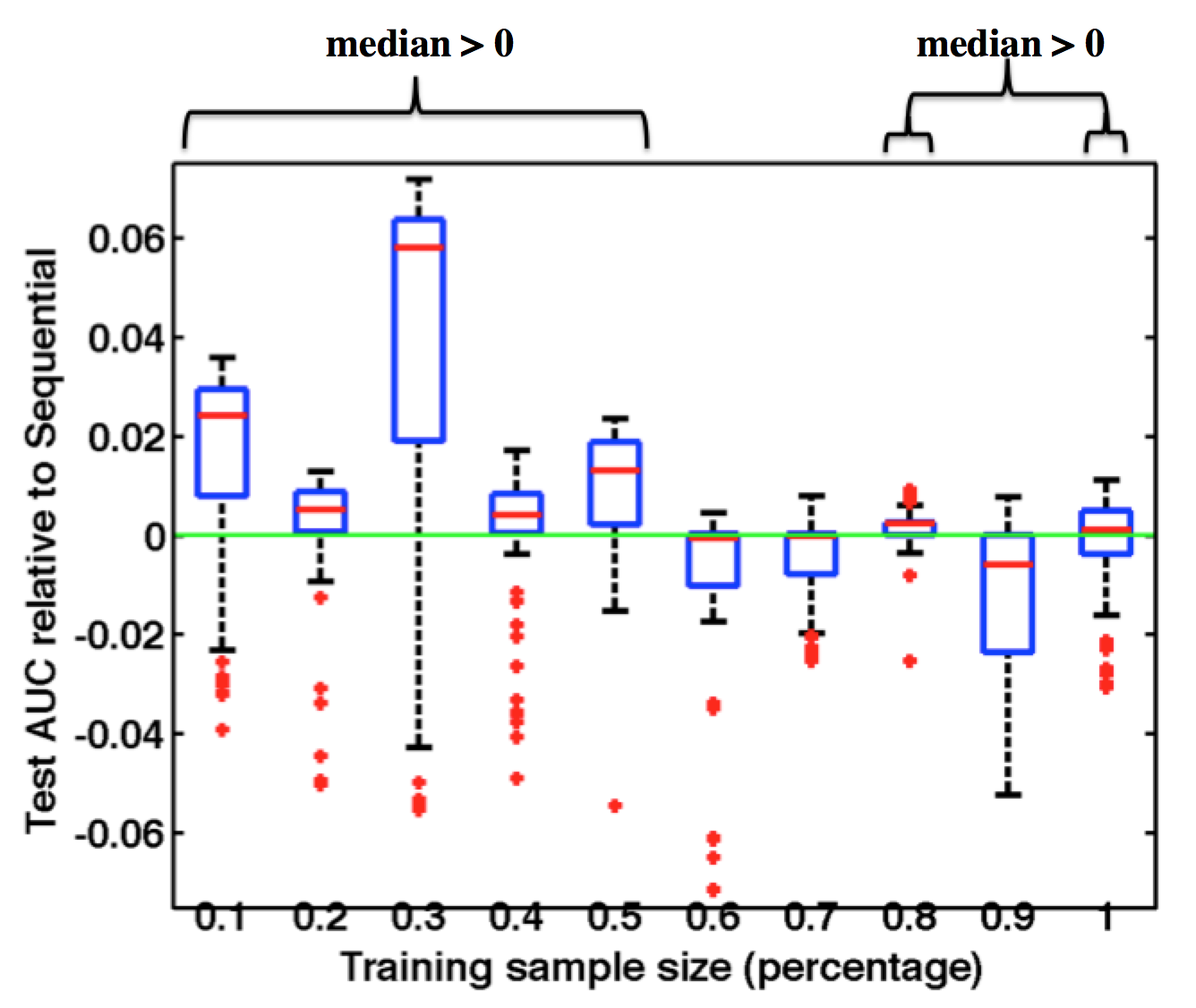}
     \label{fig:test_performance_cost2}
     \caption{Performance of the models output by the two processes on randomly generated decision problems at various training sample sizes with Cost 2 as the routing cost. The evaluation is on a separate held out test set. The green solid line is the zero mark. The box-plots at each training sample size represent the distribution of performances (relative AUC) of the models obtained by the simultaneous process.}
\end{figure}

\textcolor{black}{
\subsection{Scalability of MLOC for Routing}
In this experiment, we varied the size of the training sample and decision problem and characterized their effect on time to obtain a solution. All experiments were carried out in a cluster environment (128-256GB RAM, 16-32 core machines).
}

\textcolor{black}{
In the first case, we analyzed the effect of training sample size when the decision problem size was fixed to 7 nodes. In particular, we generated 100 seven node decision problems for each of the 10 training sample sizes (varying from 10\% to 100\% of the original) and solved the corresponding MINLPs using the AM method discussed in Section \ref{subsec:MINLPsolve}. As discussed before, a decision problem was created by randomly picking a set of seven nodes and computing the distances between them. Additionally, the $C_2$ parameter was set using 5-fold cross validation. A fixed value of $C_1$ was also chosen a-priori. Thus a total of 1000 MINLPs were solved for each Cost 1 and Cost 2. Figures \ref{fig:scaling_train_cost_type_1} and \ref{fig:scaling_train_cost_type_2} show the box plots for the time taken in seconds to solve each simultaneous process problem for Cost 1 and Cost 2 respectively.
From the figures, we can infer that as the training sample size increases, the time taken to solve the MINLP increases only mildly for both cost options. This is because the AM method can efficiently scale with the number of examples.
\begin{figure}
     \centering
     \subfigure[]{
     \includegraphics[width=.45\textwidth]{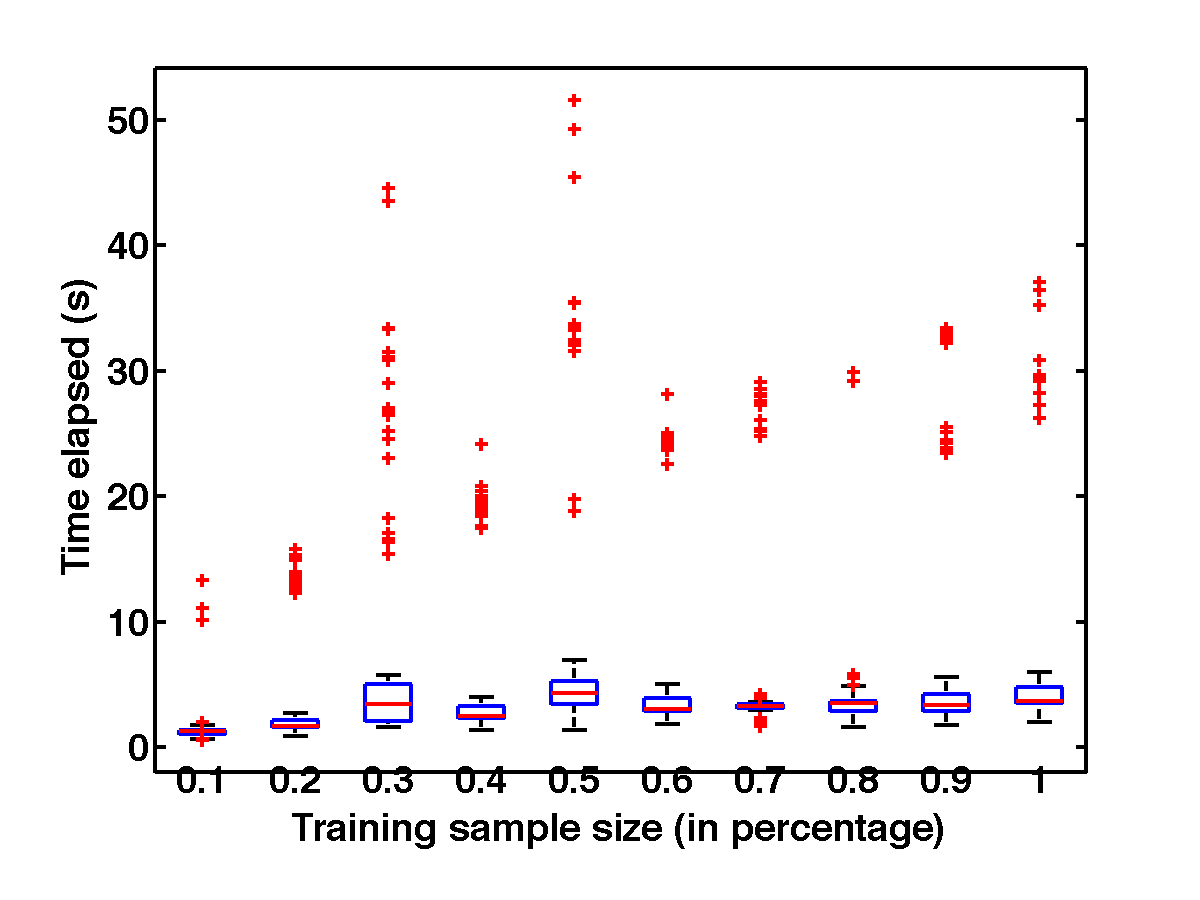}
      \label{fig:scaling_train_cost_type_1}
     }
     \subfigure[]{
     \includegraphics[width=.45\textwidth]{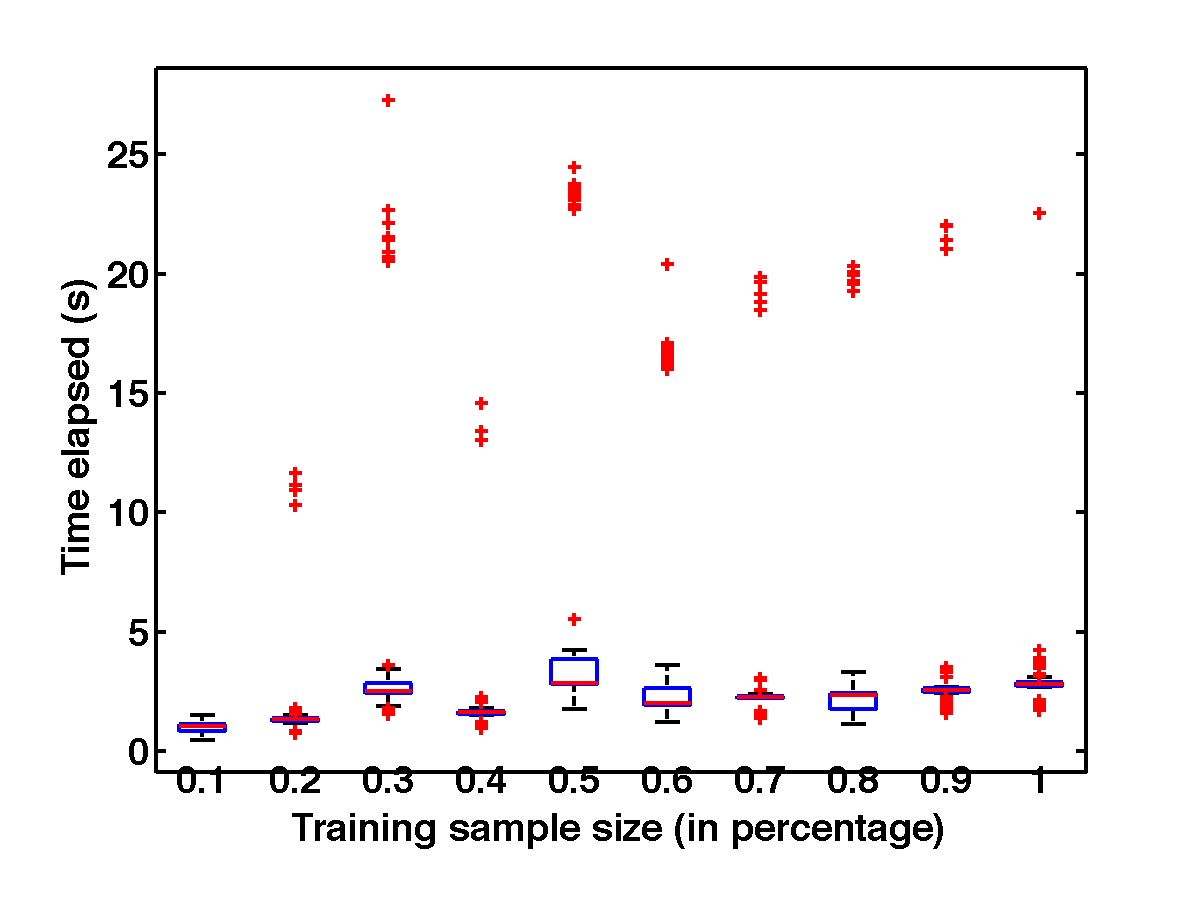}
      \label{fig:scaling_train_cost_type_2}
     }
     \caption{Left: Boxplot of times taken to solve randomly generated 7 node decision problems for various training sample sizes (from 10\% to 100\% of the original), when Cost 1 is used. For each training sample size, we solved the simultaneous process for 100 random decision problems and recorded the times. As shown, the time for solving the simultaneous process depends mildly on the size of the training sample size. Right: Boxplot of times taken to solve randomly generated 7 node decision problems for various training sample sizes when Cost 2 is used.}
\end{figure}
}

\textcolor{black}{
In the second case, we analyzed the effect of decision problem size. In particular, we generated 100 decision problems for node sizes $M= 7,8,9,10,11,12,13$ {and 10 decision problems for node size $M=15$. We} solved the MINLPs of Equations (\ref{eqn:learning1}) and (\ref{eqn:learning2}) using the AM method. Similar to the previous experiment, a decision problem of a given size was created by randomly picking a set of nodes and computing the distances between them. The $C_2$ parameter was set using 5-fold cross validation.
The MINLPs were then solved for a fixed value of $C_1$ chosen a-priori. Thus a total of 710 MINLPs were solved for each Cost 1 and Cost 2. Figures \ref{fig:scaling_cost_type_1} and \ref{fig:scaling_cost_type_2} show the box plots for the time taken {in seconds (in log scale) to solve each simultaneous process problem for Cost 1 and Cost 2 respectively.}
From the figures, we can infer that as the decision problem size ($M$ nodes) increases, the time taken to solve the MINLP increases exponentially for both cost models. As mentioned earlier, this is because TRP - and generally routing - problems are hard. One needs to solve the TRP anyway, regardless of whether the sequential or simultaneous process is used, to determine the route.
\begin{figure}
     \centering
     \subfigure[]{
     \includegraphics[width=.45\textwidth]{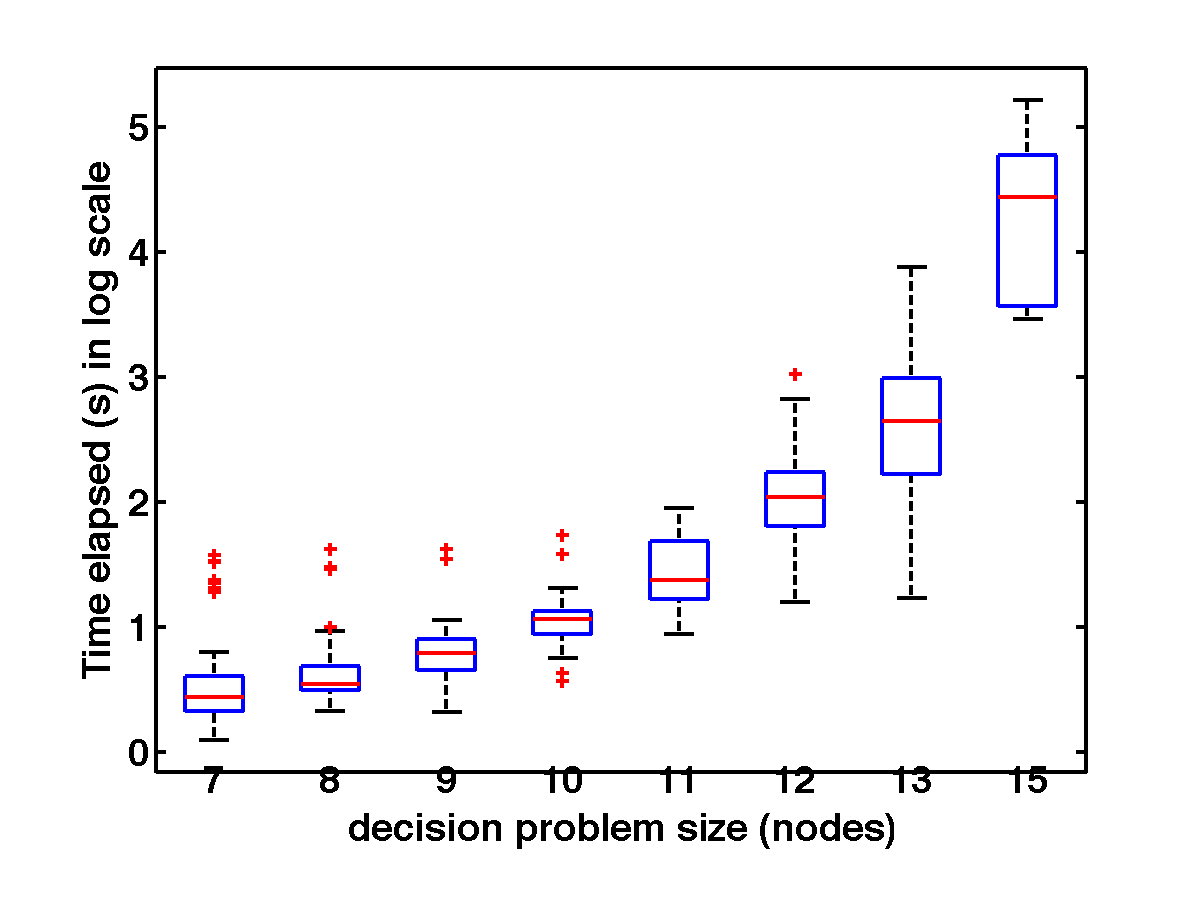}
      \label{fig:scaling_cost_type_1}
     }
     \subfigure[]{
     \includegraphics[width=.45\textwidth]{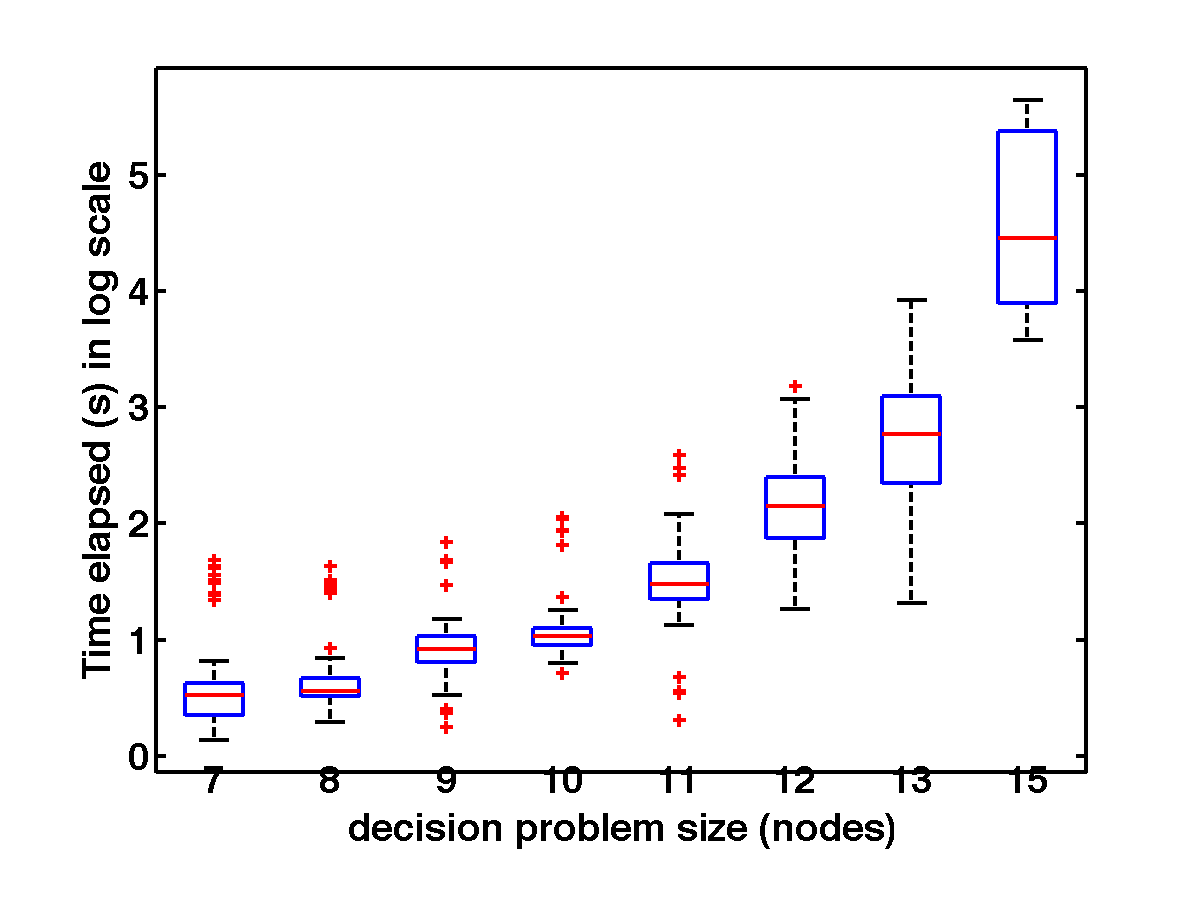}
      \label{fig:scaling_cost_type_2}
     }
     \caption{Left: Boxplot of times taken to solve the randomly generated decision problems for various values of $M$, the number of decision problem nodes, when Cost 1 is used. For each decision problem size (varying from 7 to 15), we solved the simultaneous process for 100 random decision problems (10 problems for the 15 node setting) and recorded the times. As shown, the time for solving the decision problem grows exponentially in the size of the decision/routing problem (since the trend is linear in log scale). Right: Boxplot of times taken to solve the randomly generated decision problems for various values of $M$ when Cost 2 is used.}
\end{figure}
}
\textcolor{black}{
\begin{remark} 
A note on the performance of other methods (Method 1 and Method 3): For a given $C_1$, the computation times to solve a typical problem with $\sim 23K$ examples in training and 6, 7, 8, or 10 nodes for the routing problem are about 30, 130, 140, 240 seconds respectively using Method 2 (NM). NM took $\sim 1000$ iterations to reach a solution where each iteration involved solving a weighted TRP subproblem within $\sim 2$ seconds. The computation times for solving the MINLP formulation given in (\ref{eqn:learning1}) directly (Method 1) for a given $C_{1}$ were $\sim 100$ times slower. Since the computation times for Method 2 (AM) were the best among the three, we used it  to benchmark scalability of MLOC for our application.
\end{remark}
}

\section{Generalization Bound}\label{sec:generalizationbound}
We initially introduced the failure cost regularization term in order to find scenarios where the data would support low-cost (more actionable) repair routes.
From a learning theoretic point of view, incorporating regularization reduces the size of the hypothesis space and may thus promote generalization. In our case, we can think of decision makers having prior knowledge about how much it should cost for an optimal routing solution. This information should constrain the size of the hypothesis space via the parameter $C_{1}$. Increasing $C_1$ may thus assist in predicting failure probabilities.
In what follows, we will provide a generalization bound for the 
MLOC framework, and specifically for the ML\&TRP.

We seek to bound the true risk   
$\Rtrue (f_{\lambda}):={E}_{(x,y)\sim\mu_{\mathcal{X}\times \mathcal{Y}}} l(f_{\lambda}(x),y)$ with empirical risk $\Rempirical(f_{\lambda},\{x_{i},y_{i}\}_1^m)= \frac{1}{m}\sum_{i=1}^m l(f_{\lambda}(x_i),y_i)$ plus a complexity term capturing the size of the hypothesis space. Here $l: f_{\lambda}(\mathcal{X})\times\mathcal{Y}\rightarrow \R$ is logistic loss, instance $(x,y)$ is drawn from an unknown distribution $\mu_{\mathcal{X}\times\mathcal{Y}}$ and the initial hypothesis space is $\mathcal{F}:= \{f_{\lambda}: f_{\lambda}(x) = \lambda \cdot x, \lambda \in \R^d, \|\lambda\|_{2} \leq B_{b}\}$. 

\textcolor{black}{
\subsection{Hypothesis sets for Cost 1 and Cost 2}}

Consider the ML\&TRP with Cost 1 in (\ref{eqn:learning1}).
The hypothesis space for the ML\&TRP is smaller than $\mathcal{F}$, since we have also the constraint on the failure cost.
Replacing the Lagrange multiplier $C_1$ with an explicit constraint on the failure cost (\ref{cost1def}), we have that for the ML\&TRP, $f_{\lambda}$ is subject to the failure cost constraint: 
$\min_{\pi} \sum_{i=1}^M p(\tilde{x}_{\pi(i)}) L_\pi(\pi(i)) \leq \Cbudget$, 
where $\Cbudget$ is inversely related to $C_1$, controlling a ``budget" for the failure cost. This gives us the restricted hypothesis space:
\begin{eqnarray*}\mathcal{F}_0 := \left\{f_{\lambda} : f_{\lambda}\in \mathcal{F},  \min_{\pi \in \Pi}\sum_{i=1}^M L_{\pi}({\pi(i)}) \frac{1}{1+ e^{-f_{\lambda}(\tilde{x}_{\pi(i)})}} \leq \Cbudget\right\}.
\end{eqnarray*}

\textcolor{black}{
Even though $\mathcal{F}_{0}$ is smaller than $\mathcal{F}$, it is difficult to construct a tight bound on its covering number. So we enlarge $\mathcal{F}_{0}$ just enough so that a bound on its covering number can be calculated. In particular, we will enlarge the set $\mathcal{F}_{0}$ to the set $\mathcal{F}_{2}$.
We define set $\mathcal{F}_2$ parametrized by a vector $\abudget\in\R^d$ as follows:
\[
\mathcal{F}_2 := \left\{f_{\lambda}: f_{\lambda} \in \mathcal{F}, \abudget \cdot \lambda \leq 1 \right\},
\]
 where vector $\abudget$ is a function of $\Cbudget$, the graph and the unlabeled data $\{\tilde{x}_{i}\}_{i}$.
 }
 
 \textcolor{black}{$\mathcal{F}_2$ is the intersection of the ball $\mathcal{F}$ with the halfspace defined by $\abudget$; it is a ball that is missing a spherical cap. 
The vector $\abudget$ will capture the effect of $\Cbudget$ in such a way that $\mathcal{F}_0\subset\mathcal{F}_2$, which we will show within the proof of the Theorem \ref{TheoremMain}. $\mathcal{F}_2$ is the space whose complexity we will bound, again within the proof of Theorem \ref{TheoremMain}.
} 

\textcolor{black}{We will now define the vector $\abudget$ in terms of $\Cbudget$ and provide a proof later.}
Let $d_i$ be the shortest distance from the starting node (node 1) to node $i$ for $i=2,..,M$ and $d_1$ be the length of the shortest tour that visits all the nodes and returns to node 1. This means $d_i\leq  L_{\pi}({i}); i=1,...,M$ with equality if the physical graph can be embedded into 1-dimensional Euclidean space. The vector $\abudget$ is then related to $\Cbudget$ defined elementwise as:
\begin{eqnarray}
\abudget^j = \frac{1}{\Cbudget-\azero}\left(\gprime \right)\left(\sum_id_i\tilde{x}^j_i\right) \textrm{ for } j=1,..,d	\label{eqn:abudget}\\
\textrm{where }\azero= \left(B_{b}X_{b} \;\gprime+\frac{1}{1+e^{B_{b}X_{b}}}\right)\sum_i d_i. \nonumber
\end{eqnarray}

\textcolor{black}{
\begin{remark} \textbf{(Defining $\mathcal{F}_{0}, \mathcal{F}_{2}$ and $\abudget$ for Cost 2)}: 
The definitions of $\mathcal{F}_{0}$ and $\mathcal{F}_{2}$ can be easily adapted to Cost 2 in (\ref{eqn:learning2}) of the ML\&TRP. Here too, the hypothesis space for the ML\&TRP is smaller than $\mathcal{F}$ because of the constraint on the failure cost. Again replacing the Lagrange multiplier $C_1$ with an explicit constraint on the failure cost, we have that for the ML\&TRP, $f_{\lambda}$ is subject to the failure cost constraint: 
$\min_{\pi} \sum_{i=1}^M \log(1+e^{\lambda \cdot \tilde{x}_{\pi(i)}}) L_\pi(\pi(i)) \leq \Cbudget$, 
where $\Cbudget$ is inversely related to $C_1$, controlling a ``budget" for the failure cost. This gives us the restricted hypothesis space:
\begin{eqnarray*}\mathcal{F}_0 := \{f_{\lambda} : f_{\lambda}\in \mathcal{F},  \min_{\pi \in \Pi}\sum_{i=1}^M L_{\pi}({\pi(i)}) \log(1+ e^{f_{\lambda}(\tilde{x}_{\pi(i)})}) \leq \Cbudget\}.
\end{eqnarray*}
We can again enlarge this class of functions just enough so that a bound on the covering number of $\mathcal{F}_0$ can be calculated. The enlarged set $\mathcal{F}_{2}$ will have the same form as for Cost 1 except for a different definition of $\abudget$ (we will derive this later):
\begin{eqnarray}
\abudget^j = \frac{1}{\Cbudget-\azero}\left(\gprimeCtwo \right)\left(\sum_id_i\tilde{x}^j_i\right) \textrm{ for } j=1,..,d	\label{eqn:abudget2}\\
\textrm{where }\azero= \left(B_{b}X_{b} \;\gprimeCtwo+\log(1+e^{-B_{b}X_{b}})\right)\sum_i d_i. \nonumber
\end{eqnarray}
Since Cost 2 can be handled in the same way as Cost 1, we will focus on Cost 1 for the rest of this section.
\end{remark}
}

\textcolor{black}{
\subsection{Main Generalization Result}
Recall that we would like to establish that generalization can depend on $\Cbudget$. The following theorem shows this explicitly. $\Cbudget$ enters the bound through the vector $\abudget$.
}
\begin{theorem}
\label{TheoremMain} (\textbf{Main Result})  Let $\mathcal{X} = \{x \in \mathbb{R}^{d}: \|x\|_{2} \leq X_{b}\}$, $\mathcal{Y}=\{-1,1\}$.  Let $\mathcal{F}_0$ be defined as above with respect to $\{\tilde{x}_i\}_{i=1}^M$, $\tilde{x}_i\in\mathcal{X}$ (not necessarily random) and a corresponding physical graph. Let $\{x_{i},y_{i}\}_{i=1}^{m}$ be a sequence of $m$ instances 
drawn independently according to unknown 
distribution
$\mu_{\mathcal{X}\times \mathcal{Y}}$ and 
$M_{\textrm{bound}}:= B_{b}X_{b} + \log 2$.
For any $\epsilon > 0$,
\begin{eqnarray*}
\lefteqn{P\Big(\exists f \in \mathcal{F}_0: |\Rempirical(f_{\lambda},\{x_{i},y_{i}\}_{1}^{m}) - \Rtrue(f_{\lambda})| > \epsilon\Big)}\\ &\leq& 4\alpha(d,\abudget(\Cbudget))\left(\frac{32B_{b}X_{b}}{\epsilon} +1\right)^{d} \exp\left(\frac{-m\epsilon ^2}{128M_{\textrm{bound}}^2}\right),
\end{eqnarray*}
where $\alpha(d,\abudget(\Cbudget))$ is equal to
\begin{eqnarray}
\alphaa \label{alphab}\\ 
\textrm{or equivalently, } \alphab \label{alphaa}
\end{eqnarray}
and where $_{2}F_{1}(a,b;c;d)$ and $I_{x}(a,b)$ are the hypergeometric function and the regularized incomplete beta functions respectively.
\end{theorem}

\textcolor{black}{
The term $\alpha(d,\abudget(\Cbudget))$ comes directly from formulae for the volume of spherical caps.}
As $\Cbudget$ decreases, the norm $\|\abudget\|_2$ increases, and thus $\|\abudget\|_2^{-1}$ decreases, (\ref{alphab}) and (\ref{alphaa}) decrease, and the whole bound decreases. This is the mechanism by which decreasing $\Cbudget$ may improve generalization ability. 

\textcolor{black}{
Theorem \ref{TheoremMain} is specific to the ML\&TRP because $\mathcal{F}_{0}$ was defined based on the ML\&TRP and $\abudget$ was defined in (\ref{eqn:abudget}) for Cost 1 and (\ref{eqn:abudget2}) for Cost 2. 
}

\textcolor{black}{
The technique of Theorem \ref{TheoremMain} applies much more broadly than the ML\&TRP. In fact, we can derive a general bound that applies to any problem with a similar hypothesis space constraint. Specifically, the hypothesis space should be bounded by the intersection of a ball with a half-space.
}

\begin{corollary}(\textbf{Bound for General MLOC Framework}) Consider any operational cost constraint such that the hypothesis space lies within $\mathcal{F}_{2}$ defined by $\mathcal{F}_{2} = \{f_{\lambda} \in \mathcal{F}: \abudget \cdot \lambda \leq 1\}$ for some $\abudget \in \mathbb{R}^{d}$. Then, for any $\epsilon > 0$,
\begin{eqnarray*}
\lefteqn{P\Big(\exists f \in \mathcal{F}_{2}: |\Rempirical(f_{\lambda},\{x_{i},y_{i}\}_{1}^{m}) - \Rtrue(f_{\lambda})| > \epsilon\Big)}\\ &\leq& 4\alpha(d,\abudget)\left(\frac{32B_{b}X_{b}}{\epsilon} +1\right)^{d} \exp\left(\frac{-m\epsilon ^2}{128M_{\textrm{bound}}^2}\right),
\end{eqnarray*}
\textcolor{black}{
where $\alpha(d,\abudget)$ equals
\begin{eqnarray*}
\alphaa \\ 
\textrm{or equivalently, } \alphab 
\end{eqnarray*}
and where $_{2}F_{1}(a,b;c;d)$ and $I_{x}(a,b)$ are the hypergeometric function and the regularized incomplete beta functions respectively.
}
\label{CorollaryMain}
\end{corollary}

\textcolor{black}{
The $\alpha(d,\abudget)$ is influenced by our belief on the operational cost. Thus, by being able to specify something about the operational cost, we are able to have a better guarantee on generalization. In the case where we are not able to specify anything about the operational cost, the quantity $\alpha(d,\abudget)$ is equal to $1$ giving us the standard generalization result for norm constrained linear function classes.
}

\subsection{Proof}

\textcolor{black}{
The proof outline is as follows. We will construct two classes, $\mathcal{F}_1$ and $\mathcal{F}_2$ that are slightly larger than $\mathcal{F}_0$, but smaller than $\mathcal{F}$ when $\Cbudget$ is small enough. Then we will use a volumetric argument to bound the covering number of $\mathcal{F}_2$, which uses the volumes of spherical caps; the idea is to show that the value of $\Cbudget$ affects the volume of the hypothesis space, and thus the covering number.  The covering number bound is then applied to a uniform bound of \cite{pollard84} to obtain a generalization bound. 
The fact that the covering number of  $\mathcal{F}_2$ can be below that of $\mathcal{F}$ indicates that using functions from $\mathcal{F}_2$ may provide improvements in generalization over using the full set $\mathcal{F}$.
}

Let us lead up to the proof of Theorem \ref{TheoremMain}.
\begin{definition} Let $A \subseteq X$ be an arbitrary set and $(X, \dist)$ a (pseudo) metric space. Let $|\cdot|$ denote set size.
\begin{itemize}
\item For any $\epsilon > 0$, an \underline{$\epsilon$-cover} for $A$ is a finite set $U \subseteq X$ (not necessarily $ \subseteq A$) s.t. $ \forall x \in A, \exists u \in U$ with $\dist(x, u) \leq \epsilon$.
\item $A$ is totally bounded if $A$ has a finite $\epsilon$-cover for all $\epsilon > 0$. The \underline{\textit{covering number}} of $A$ is $N(\epsilon,A,\dist) := \inf_{U \in \mathcal{U}} |U|$ where $\mathcal{U}$ is the set of all $\epsilon$-covers for $A$. 
\item A set $R \subseteq X$ is $\epsilon$-separated if $\forall x,y \in R, \dist(x,y) > \epsilon$.  The \underline{\textit{packing number}} $M(\epsilon,A,\dist) := \sup_{R \in \mathcal{R}} |R|$, where $\mathcal{R}$ is the set of all $\epsilon$-separated subsets of $A$.
\end{itemize}
\end{definition}

\textcolor{black}{Consider Cost 1}.
Since, for any collection of values $p(\tilde{x}_i)\geq 0, \sum_i d_i p(\tilde{x}_i)\leq \sum_i L_{\pi}(i) p(\tilde{x}_i) \leq \Cbudget$, the class of functions which obey the constraint $\sum_i d_i p(\tilde{x}_i)\leq \Cbudget$ is larger than the class obeying $\sum_i L_{\pi}(i) p(\tilde{x}_i) \leq \Cbudget$. That is, $\mathcal{F}_0\subseteq \mathcal{F}_1$ where 
\[\mathcal{F}_1 := \left\{f_{\lambda}: f_{\lambda} \in \mathcal{F}, \sum_{i=1}^{M}d_i \frac{1}{1+ e^{-f_{\lambda}(\tilde{x}_{i})}} \leq \Cbudget \right\}.\]
As long as $\Cbudget \leq \sum_{i=1}^{M}d_{i} $, the constraint in $\mathcal{F}_1$ is not vacuous. The choice of the vector $\abudget$ ensures that $\mathcal{F}_1$ is a subset of $\mathcal{F}_2$ as we will prove below.

\begin{lemma}
\label{LemmaRelaxationCovering}  (\textbf{$\mathcal{F}_{0}$ is contained in $\mathcal{F}_{2}$})
\[N(\epsilon, \mathcal{F}_0,\|\cdot\|_{L_{2}(\mu_{\mathcal{X}}^{m})}) \leq N(\epsilon, \mathcal{F}_1,\|\cdot\|_{L_{2}(\mu_{\mathcal{X}}^{m})}) \leq N(\epsilon, \mathcal{F}_2,\|\cdot\|_{L_{2}(\mu_{\mathcal{X}}^{m})}).\]
\end{lemma}
\begin{proof} It is sufficient to show $\mathcal{F}_{0} \subseteq \mathcal{F}_{1} \subseteq \mathcal{F}_{2}$.
The first inequality was discussed earlier;
since $d_{i} = \inf_{\pi \in \Pi}L_{\pi}({i})$, this implies:
\begin{eqnarray*}
\sum_{i=1}^{M}d_{i}p(\tilde{x}_{i}) \leq \sum_{i=1}^{M}L_{\pi}({i})p(\tilde{x}_{i}) \leq \Cbudget \Rightarrow \mathcal{F}_{0} \subseteq \mathcal{F}_{1}.
\end{eqnarray*}

We now show $\mathcal{F}_{1} \subseteq \mathcal{F}_{2}$. We first lower bound $p(\tilde{x}_i)$ by a line with slope $m_{1} := \gprime$ and intercept $m_{0} := B_{b}X_{b} \gprime +\frac{1}{1+e^{B_{b}X_{b}}}$ such that $m_{1}f_{\lambda}(\tilde{x}_i)+m_0\leq p(\tilde{x}_i)$ within the function range $[-B_{b} X_{b},B_{b} X_{b}]$.

This leads to the definition of $\abudget$ as we show now:
\begin{eqnarray}
\label{ctilde}
 &\sum_i d_i p(\tilde{x}_i)\geq \sum_i d_i (m_{1} (\lambda \cdot \tilde{x}_i)+m_0)
 = \tilde{a} \cdot \lambda+\azero,\\
\textrm{where } & \tilde{a}^j:=m_{1}\left(\sum_id_i\tilde{x}^j_i\right)= \frac{e^{B_{b}X_{b}}}{(1+e^{B_{b}X_{b}})^2}\left(\sum_id_i\tilde{x}^j_i\right) \textrm{ for } j=1,...,d
\label{eqn:atilde}\\
\textrm{and } & \azero=m_0\sum_i d_i = \left(B_{b}X_{b} \;\frac{e^{B_{b}X_{b}}}{(1+e^{B_{b}X_{b}})^2}+\frac{1}{1+e^{B_{b}X_{b}}}\right)\sum_i d_i.\nonumber \\
\textrm{Thus } & \forall \lambda\in\mathcal{F}_1, \tilde{a}\cdot \lambda+\azero 	\leq		   \sum_{i=1}^{M}d_{i}p(\tilde{x}_{i}) 				\leq \Cbudget,
\end{eqnarray}
which implies $\tilde{a}\cdot \lambda \leq \Cbudget - \azero$ or equivalently, $\frac{1}{\Cbudget - \azero}\tilde{a}\cdot \lambda 	\leq 1$.

This allows us to define $\abudget$ using (\ref{eqn:atilde}) as 
\[
\abudget^j = \frac{1}{\Cbudget-\azero}\left(\gprime \right)\left(\sum_id_i\tilde{x}^j_i\right) \textrm{ for } j=1,..,d,
\]
which is the same as (\ref{eqn:abudget}). This vector is such that the set $\mathcal{F}_{2}$ is larger than $\mathcal{F}_1$.
\qed
\end{proof}
\textcolor{black}{
\begin{remark} \textbf{(Deriving $\abudget$ for Cost 2)}: 
The above lemma can be adapted to Cost 2 to give the corresponding $\abudget$ that we had defined earlier. In particular, for any collection of values $\log(1+e^{\lambda \cdot \tilde{x}_i})\geq 0$ for all $i$, 
\[
\sum_i d_i \log(1+e^{\lambda \cdot \tilde{x}_i}) \leq \sum_i L_{\pi}(i) \log(1+e^{\lambda \cdot \tilde{x}_i}).
\]
Thus the class of functions that obey the constraint $\sum_i d_i \log(1+e^{\lambda \cdot \tilde{x}_i}) \leq \Cbudget$ is larger than the class obeying $\sum_i L_{\pi}(i) \log(1+e^{\lambda \cdot \tilde{x}_i}) \leq \Cbudget$, which is $\mathcal{F}_{0}$. $\mathcal{F}_{1}$ will be the set corresponding to the former constraint: 
\[
\mathcal{F}_{1} := \left\{f_{\lambda}\in \mathcal{F}:  \sum_{i=1}^M d_{i} \log(1+ e^{\lambda \cdot \tilde{x}_{i}}) \leq \Cbudget\right\}.
\]
 We now define $\mathcal{F}_{2}$ and $\abudget$ as follows.
We can also see that $\log(1+e^{\lambda \cdot \tilde{x}_i})$ can be lower bounded by a line with slope $m_{1} := \gprimeCtwo$ and intercept $m_{0}:= B_{b}X_{b} \gprimeCtwo + \log(1+e^{-B_{b}X_{b}})$ in the function range $[-B_{b} X_{b},B_{b} X_{b}]$ giving us the definition of $\abudget$ for Cost 2 as follows:
\begin{eqnarray*}
\Cbudget \geq &\sum_i d_i \log(1+e^{\lambda \cdot \tilde{x}_i})\geq \sum_i d_i (m_{1} (\lambda \cdot \tilde{x}_i)+m_0)
 = \tilde{a} \cdot \lambda+\azero,\\
\textrm{where} & \tilde{a}^j:=m_{1}\left(\sum_id_i\tilde{x}^j_i\right)= \gprimeCtwo\left(\sum_id_i\tilde{x}^j_i\right) \textrm{ for } j=1,...,d\\
\textrm{and } & \azero=m_0\sum_i d_i = \left(B_{b}X_{b} \;\gprimeCtwo+ \log(1+e^{-B_{b}X_{b}})\right)\sum_i d_i.\nonumber
\end{eqnarray*}
Thus,  $\frac{1}{\Cbudget - \azero}\tilde{a}\cdot \lambda 	\leq 1$, and since we wanted to have $\abudget\cdot\lambda \leq 1$ we define $\abudget$ element-wise as:
\[
\abudget^j = \frac{1}{\Cbudget-\azero}\left(\gprimeCtwo \right)\left(\sum_id_i\tilde{x}^j_i\right) \textrm{ for } j=1,..,d.
\]
Note that we have produced two $\abudget$ vectors for each of the two costs: Cost 1 and Cost 2 above. 
\end{remark}
}

Let $B(0,B_{b}) := \{\lambda: \lambda \in \R^{d}, \|\lambda\|_{2} \leq B_{b}\}$. Let the half space corresponding to $\mathcal{F}_2$ be $H_{\|\abudget\|_{2}^{-1}} := \{\lambda: \abudget \cdot \lambda \leq 1\}$.  The lemma below relates covering numbers of $\mathcal{F}$ and $\mathcal{F}_2$ in function space to covering numbers of $B(0,B_{b})$ and $B(0,B_{b})\cap H_{\|\abudget\|_{2}^{-1}}$ in $\R^d$.

\begin{lemma} 
\label{LemmaL2tol2} (\textbf{Relating covering numbers in $\|\cdot\|_{L_{2}(\mu_{\mathcal{X}}^{m})}$  to $\|\cdot\|_{2}$})
\begin{enumerate}
\item[a.] $\sup_{\mu_{\mathcal{X}}^{m}}N(\epsilon,\mathcal{F},\|\cdot\|_{L_{2}(\mu_{\mathcal{X}}^{m})}) \leq N(\epsilon/X_{b},B(0,B_{b}),\|\cdot\|_{{2}})$, and
\item[b.] $\sup_{\mu_{\mathcal{X}}^{m}}N(\epsilon,\mathcal{F}_{2},\|\cdot\|_{L_{2}(\mu_{\mathcal{X}}^{m})}) \leq N(\epsilon/X_{b},B(0,B_{b})\cap H_{\|\abudget\|_{2}^{-1}}, \|\cdot\|_{{2}})$.
\end{enumerate}
\end{lemma}
\textcolor{black}{
\begin{proof}  Each element $f\in\mathcal{F}$ corresponds to at least one element of $B(0,B_{b})$ by definition of $\mathcal{F}$. Choose any distribution $\mu_{\mathcal{X}}^{m}$. Consider two elements $\lambda_{f},\lambda_{g} \in B(0,B_{b})$ corresponding to functions $f,g \in \mathcal{F} \subset L_{2}(\mu_{\mathcal{X}}^{m})$. Then,
\begin{eqnarray*}
\|f-g\|^{2}_{L_2(\mu_{\mathcal{X}}^{m})} &=& \frac{1}{m}\sum_{i=1}^{m}(f(x_{i})-g(x_{i}))^{2}\\
&=& \frac{1}{m}\sum_{i=1}^{m} \left((\lambda_{f}-\lambda_{g}) \cdot x_{i}\right)^{2}\\
&\leq& \frac{1}{m}\sum_{i=1}^{m}\|\lambda_{f}-\lambda_{g}\|_{2}^{2}\|x_{i}\|_{2}^{2} \textrm{  (Cauchy-Schwarz to each term)}\\
&\leq&  \|\lambda_{f}-\lambda_{g}\|_{2}^{2}\left(\frac{1}{m} \sum_{i=1}^{m} X_{b}^{2}\right) \textrm{  (since } \sup_{x \in \mathcal{X}}\|x\|_{2} \leq X_{b} )\\
&=& \|\lambda_{f}-\lambda_{g}\|_2^{2}X_{b}^{2}.
\end{eqnarray*}
Consider a minimal $\epsilon/X_{b}$-cover $\{\lambda_{r}\}_r$ for $B(0,B_{b})$ where $\lambda_r$ corresponds to a function $r\in \mathcal{F}$. Then by definition, $\forall \lambda \in B(0,B_{b}), \exists \lambda_{r}: \|\lambda-\lambda_{r}\|_2 \leq \epsilon/X_{b}$. Thus, picking any two such elements $\lambda_{f},\lambda_{g}$ in a ball of radius $\epsilon/X_{b}$ around $\lambda_{r}$, we see that, the corresponding functions $f,g$ belong to a ball of radius $\epsilon$ measured using distance in $L_{2}(\mu_{\mathcal{X}}^{m})$ by the inequality above.
The centers of these $\epsilon$-balls in $L_{2}(\mu_{\mathcal{X}}^{m})$ form an $\epsilon$-cover for $\mathcal{F}$. 
The size of this set is equal to $N(\epsilon/X_{b},B(0,B_{b}),\|\cdot\|_{2})$ (which is the size of $\epsilon/X_{b}$-cover for $B(0,B_{b})$). The size of the minimal $\epsilon$-cover of $\mathcal{F}$ will be less than or equal to this size. Hence,
$N(\epsilon,\mathcal{F},\|\cdot\|_{L_{2}(\mu_{\mathcal{X}}^{m})}) \leq N(\epsilon/X_{b},B(0,B_{b}),\|\cdot\|_{2})$.
Taking a supremum over all $\mu_{\mathcal{X}}^{m}$, we obtain the first inequality of the lemma. The same argument also works for the second inequality. \qed
\end{proof}
}


Because of rotational symmetry of $B(0,B_{b})$, the volume cut off by a hyperplane $\abudget \cdot \lambda = 1$ from $B(0,B_{b})$ is determined only by its distance from the origin, which is $1/\|\abudget\|_{2}$. Such a portion (or its complement, if smaller) of a ball obtained from slicing the ball with a hyperplane is called a spherical cap. It can be parameterized by the distance of its (hyper)plane base from the center of the ball as shown below.  For notation, let the volume of a set $A \subset \mathbb{R}^{d}$ be represented as $Vol(A)$. For example, $Vol(B_{1}) = \frac{\pi^{d/2}}{\Gamma\left[d/2+1\right]}$.
\begin{lemma}  
\label{LemmaSphericalCap} (\textbf{Volume of spherical caps}) Let the volume of ball $B(0,B_{b})$ in $\mathbb{R}^{d}$ be denoted as $Vol(B(0,B_{b}))$. Given a $d$-dimensional vector $a$, let $z = \|a\|_{2}^{-1}$ be a number and $H_{z} = \{\lambda: a \cdot \lambda \leq 1\}$ be a half space parameterized by $z$. Let the spherical cap be denoted by $B(0,B_{b})\cap H_{z}'$ where the cap is at a distance $z$ (measured from the base of the cap to the center of the ball), and $H_{z}'$ represents the complement half space ($H_{z} \cup H_{z}' = \mathbb{R}^{d}$). Then, $Vol(B(0,B_{b})\cap H_{z}')/Vol(B(0,B_{b}))  $ is equal to two expressions:
\begin{eqnarray*}
 &\left( \frac{1}{2} - \frac{z}{B_{b}}  \frac{\Gamma\left[1+\frac{d}{2}\right]}{\sqrt{\pi}\Gamma\left[\frac{d+1}{2}\right]} {}_{2}F_{1}\left(\tfrac{1}{2},\tfrac{1-d}{2};\tfrac{3}{2};\left(\tfrac{z}{B_{b}}\right)^{2}\right)\right)
=&  \frac{1}{2} I_{1-z^{2}/B_{b}^{2}} \left(\frac{d+1}{2}, \frac{1}{2} \right),
\end{eqnarray*} 
where $_{2}F_{1}(a,b;c;d)$ and $I_{x}(e,f)$ are the hypergeometric and regularized incomplete beta functions respectively.
\end{lemma}
\begin{proof} See \cite{li11} and references therein. 
\end{proof}


Next, we need the relationship between packing numbers and covering numbers to prove Theorem \ref{TheoremVolCovering}:
\begin{lemma}  
\label{LemmaCoveringPacking}(\textbf{Packing and covering numbers}) For every (pseudo) metric space $(X, \dist)$, $A \subseteq X$, and $\epsilon > 0$,
\begin{equation*}
N(\epsilon, A,\dist) \leq M(\epsilon, A,\dist).
\end{equation*}
\end{lemma}
\begin{proof} 
See Theorem 4 in \cite{kol61} \textcolor{black}{or Theorem 12.1 in \cite{bartlett99} for a proof of this classical result.}
\end{proof}

We use the above lemma to obtain bounds for the covering numbers of subsets of $\R^d$ which appeared in Lemma \ref{LemmaL2tol2}. 
\begin{theorem}
\label{TheoremVolCovering} 
(\textbf{Bound on Covering Numbers})
\begin{eqnarray*}
N(\epsilon/X_{b},B(0,B_{b}),\|\cdot\|_{2}) &\leq& \left(\frac{2B_{b}X_{b}}{\epsilon} +1\right)^{d},\textrm{ and}\\
N\left(\epsilon/X_{b},B(0,B_{b})\cap H_{\|a\|_{2}^{-1}},\|\cdot\|_{2}\right) &\leq& \left(\frac{Vol\left(B_{B_{b}+\frac{\epsilon}{2X_{b}}}\cap H_{\|a\|_{2}^{-1}+\frac{\epsilon}{2X_{b}}}\right)}{Vol\left(B_{B_{b}+\frac{\epsilon}{2X_{b}}}\right)}\right)\left(\frac{2B_{b}X_{b}}{\epsilon} +1\right)^{d}.
\end{eqnarray*}
\end{theorem}
\begin{proof} 
Both statements involve a volumetric argument. For a proof of the first inequality, see Section 3 of \cite{kol61} 
\textcolor{black}{or Lemma 4.10 in \cite{pisier89} or \cite{lorentz66} or Lemma 3 in  \cite{cuckersmale01}. }

To show the second part, let the volume of the complement of the spherical cap be $Vol(B(0,B_{b})\cap H_{\|a\|_{2}^{-1}})$; we need to find an upper bound for the minimal $\epsilon/X_{b}$-cover of this set. We can do that by scaling a minimal $\epsilon$-cover, which we find now. By extending the boundary of $B(0,B_{b})\cap H_{\|a\|_{2}^{-1}}$ by $\epsilon/2$ we can bound the maximal packing number $M(\epsilon,B(0,B_{b})\cap H_{\|a\|_{2}^{-1}},\|\cdot\|_{2})$ as follows:
\begin{eqnarray*}
M(\epsilon,B(0,B_{b})\cap H_{\|a\|_{2}^{-1}},\|\cdot\|_{2})&\times&Vol(B_{1}) (\epsilon/2)^{d} \leq Vol(B_{B_{b}+\epsilon/2}\cap H_{\|a\|_{2}^{-1}+\epsilon/2}).\\
\textrm{Or, }M(\epsilon,B(0,B_{b})\cap H_{\|a\|_{2}^{-1}},\|\cdot\|_{2}) &\leq& \frac{Vol\left(B_{B_{b}+\epsilon/2}\cap H_{\|a\|_{2}^{-1}+\epsilon/2}\right)}
{Vol(B_{1})}\frac{1}{(\epsilon/2)^{d}}\\
&=&  \frac{Vol\left(B_{B_{b}+\epsilon/2}\cap H_{\|a\|_{2}^{-1}+\epsilon/2}\right)}{Vol(B_{1})}\frac{1}{(\epsilon/2)^{d}} \frac{(B_{b}+\epsilon/2)^{d}}{(B_{b}+\epsilon/2)^{d}}\\
&=& \frac{Vol\left(B_{B_{b}+\epsilon/2}\cap H_{\|a\|_{2}^{-1}+\epsilon/2}\right)}{Vol(B_{B_{b}+\epsilon/2})}\frac{(B_{b}+\epsilon/2)^{d}}{(\epsilon/2)^{d}}.
\end{eqnarray*}
Again, scaling $\epsilon$ to $\epsilon/X_{b}$ and using the relationship between $N(\epsilon,A,dist)$ and $M(\epsilon,A,dist)$ in Lemma \ref{LemmaCoveringPacking} yields the second result. \qed
\end{proof}

Thus we have so far shown the relationship between covering numbers of  $\mathcal{F}_0$, $\mathcal{F}_1$, and $\mathcal{F}_2$ in terms of a certain metric in Lemma \ref{LemmaRelaxationCovering}, we have shown how those covering numbers are related to covering numbers in $\ell_2(\R^d)$ in Lemma \ref{LemmaL2tol2}, we have shown how the latter covering numbers relate to volumes in $\ell_2(\R^d)$ in Theorem \ref{TheoremVolCovering}, and we have shown how to compute one of these volumes in Lemma \ref{LemmaSphericalCap}.

To complete the proof of Theorem \ref{TheoremMain}, we will use a relation between the covering number of a class of loss functions of some set $\mathcal{G}$ and the covering number of the set $\mathcal{G}$ itself. We will also use a uniform convergence bound of \citet{pollard84}.

\begin{theorem}
\label{TheoremPollard} (\textbf{Pollard 1984})  Let $l_{\mathcal{G}}$ be a set of functions on $\mathcal{X}\times \mathcal{Y}$ with $0\leq l(f_{\lambda}(x),y) \leq M_{\textrm{bound}}, \;\; \forall l \in l_{\mathcal{G}}$ and $\forall (x,y) \in \mathcal{X}\times \mathcal{Y}$. Let $\{x_{i},y_{i}\}_{1}^{m}$ be a sequence of $m$ instances drawn independently according to $\mu_{\mathcal{X}\times \mathcal{Y}}$. Then for any $\epsilon > 0$,
\begin{eqnarray*}
\lefteqn{P(\exists l \in l_{\mathcal{G}}: |\Rempirical(f_{\lambda},\{x_{i},y_{i}\}_{1}^{m}) - \Rtrue(f_{\lambda})| > \epsilon) }\\
&\leq& 4E\left[N\left(\epsilon/16,l_{\mathcal{G}},\|\cdot\|_{L_{1}(\mu_{\mathcal{X}\times \mathcal{Y}}^{m})}\right)\right]\exp{\left(\frac{-m\epsilon ^2}{128 M_{\textrm{bound}}^2}\right)}.
\end{eqnarray*}
\end{theorem}
\begin{proof} See Theorem 24 in \cite{pollard84} \citep[also in][Theorem 1]{zhang02}. 
\end{proof}

We can relate the covering number for Pollard's loss functions set $l_{\mathcal{G}}$ to the covering number for set $\mathcal{G}$ as follows.
\begin{lemma} 
\label{LemmaLossClassFunctionClass} (\textbf{Relating $l_{\mathcal{G}}$ to $\mathcal{G}$}) If every function from function class $l_{\mathcal{G}}$ represented as $l:f(\mathcal{X})\times \mathcal{Y} \mapsto \mathbb{R}, f \in \mathcal{G}$, is Lipschitz in its first argument with Lipschitz constant $\mathcal{L}$, then the covering number of $l_{\mathcal{G}}$ is related to the covering number of $\mathcal{G}$ by
\begin{equation*}
\sup_{\mu_{\mathcal{X}\times \mathcal{Y}}^{m}}N\left(\epsilon,l_{\mathcal{G}},\|\cdot\|_{L_{1}(\mu_{\mathcal{X}\times \mathcal{Y}}^{m})}\right) \leq N\left(\epsilon/\mathcal{L},\mathcal{G},\|\cdot\|_{L_{1}(\mu_{\mathcal{X}}^{m})}\right).
\end{equation*}
\end{lemma}

\textcolor{black}{
\begin{proof} Consider two functions $f,g \in \mathcal{G}$. Let the corresponding functions in class $l_{\mathcal{G}}$ be $l_{f}=l(f(x),y)$ and $l_{g}=l(g(x),y)$.
\begin{eqnarray*}
\|l_f-l_g\|_{L_{1}(\mu_{\mathcal{X}\times \mathcal{Y}}^{m})} &=& \frac{1}{m}\sum_{i=1}^{m}|l(f(x_{i}),y_{i}) - l(g(x_{i}),y_{i})|\\
 &\leq& \frac{1}{m}\sum_{i=1}^{m} \mathcal{L} |f(x_{i}) - g(x_{i})|
 = \mathcal{L} \|f-g\|_{L_1(\mu_{\mathcal{X}}^{m})}.
\end{eqnarray*}
This implies, given $\{x_{i},y_{i}\}_{i=1}^{m}$, if $\hat{\mathcal{G}}$ is a minimal $\epsilon/\mathcal{L}$-cover of $\mathcal{G}$ in $L_1(\mu_{\mathcal{X}}^{m})$, we can construct an $\epsilon$-cover of $l_{\mathcal{G}}$ in $L_1(\mu_{\mathcal{X}\times \mathcal{Y}}^{m})$ as $ \hat{l}_{\mathcal{G}} = \{l_{f_{i}}: f_{i}\in \hat{\mathcal{G}}\} $.
The size of the minimal $\epsilon$-cover will be smaller than the size of such an $\epsilon$-cover. Taking the supremum over all empirical distributions, we get the desired result. \qed
\end{proof}
}

Theorem \ref{TheoremPollard} and Lemma \ref{LemmaLossClassFunctionClass} involve $L_1$ covering numbers, but our covering number bounds start with an $L_{2}$ metric in Lemma \ref{LemmaL2tol2}. So we need to switch from $L_{1}$ to $L_{2}$ metric. The following lemma uses the identity $\|f-g\|_{L_{1}(\mu_{X}^{m})} \leq \|f-g\|_{L_{2}(\mu_{X}^{m})}$ (true because of Jensen's inequality applied to norms) to relate the two.
\begin{lemma}
\label{LemmaL1L2Covering} 
 $N(\epsilon,A,\|\cdot\|_{L_{1}(\mu_{\mathcal{X}}^{m})}) \leq N(\epsilon,A,\|\cdot\|_{L_{2}(\mu_{\mathcal{X}}^{m})})$.
\end{lemma}
\begin{proof} See for a version, Lemma 10.5 in \cite{bartlett99}.
\end{proof}

Finally, we can prove the main result.
\begin{proof} \textit{(Of Theorem \ref{TheoremMain})}

\textcolor{black}{
In our setting, the loss function is logistic with Lipschitz constant $ \mathcal{L} = 1$ (when viewed as a function of $f(x)$). The class of loss functions is thus defined by $l_{\mathcal{F}_{0}} := \{l :f_{\lambda} \in \mathcal{F}_{0}\}$. Each $l \in l_{\mathcal{F}_{0}}$ is also non-negative and bounded as needed in the statement of Theorem \ref{TheoremPollard}.}

Starting from the expectation term on the right hand side of Theorem \ref{TheoremPollard} using $\mathcal{F}_{0}$ as $\mathcal{G}$ we get,
\begin{eqnarray*}
\lefteqn{E[N(\epsilon/16,l_{\mathcal{F}_{0}},\|\cdot\|_{L_{1}(\mu_{\mathcal{X}\times \mathcal{Y}}^{m})})]}\\
 &\leq& \sup_{\mu_{\mathcal{X}\times \mathcal{Y}}^{m}}N(\epsilon/16,l_{\mathcal{F}_{0}},\|\cdot\|_{L_{1}(\mu_{\mathcal{X}\times \mathcal{Y}}^{m})}) \textrm{ bounding expectation by supremum}
 \\
&\leq& \sup_{\mu_{\mathcal{X}}^{m}} N\left(\frac{\epsilon}{16\mathcal{L}},\mathcal{F}_{2},\|\cdot\|_{L_{2}(\mu_{\mathcal{X}}^{m})}\right) \textrm{ from Lemma \ref{LemmaLossClassFunctionClass}, \ref{LemmaL1L2Covering} and \ref{LemmaRelaxationCovering} respectively}\\
 &\leq& N\left(\frac{\epsilon}{16\cdot 1 \cdot X_{b}},B(0,B_{b})\cap H_{\|\abudget\|_{2}^{-1}} , \|\cdot\|_2\right) \textrm{ from Lemma \ref{LemmaL2tol2}, substituting $\mathcal{L} = 1$}\\
 &\leq& \left(\frac{Vol\left(B_{B_{b}+\frac{\epsilon}{32X_{b}}}\cap H_{\|\abudget\|_{2}^{-1}+\frac{\epsilon}{32X_{b}}}\right)}{Vol(B_{B_{b} + \frac{\epsilon}{32X_{b}} })}\right)\left(\frac{32B_{b}X_{b}}{\epsilon} +1\right)^{d} \textrm{ from Theorem \ref{TheoremVolCovering}}\\
 &=& \alpha(d,\abudget(\Cbudget)) \left(\frac{32B_{b}X_{b}}{\epsilon} +1\right)^{d} \textrm{ from Lemma \ref{LemmaSphericalCap}}.
\end{eqnarray*}
The above step uses the relation between spherical cap and its complement along with Lemma \ref{LemmaSphericalCap}, $
Vol\left(B(0,B_{b})\cap H_{\|\abudget\|_2^{-1}}'\right)=Vol(B(0,B_{b}))-Vol\left(B(0,B_{b})\cap H_{\|\abudget\|_2^{-1}}\right)$.

Using the derived inequality within Theorem \ref{TheoremPollard} completes the proof. \qed
\end{proof}

\textcolor{black}{
\section{Future work}\label{sec:future}
We provide several avenues for future work.
\begin{itemize}
\item \textit{Other graph applications:} The MLOC framework is a general tool that can help decision makers translate uncertainty in prediction to uncertainty in operational costs. The ML\&TRP itself is a specific application of the MLOC framework that can be applied to the power grid (as we did), but also to delivery truck routing and other physical routing problems, and can be used for more abstract routing problems such as network routing problems, where distances on the graph do not necessarily correspond to a physical distance. In the future it would be interesting to explore some of these applications.
\item \textit{Relaxing the cost constraints in the MLOC:} Our generalization bound for the ML\&TRP applied to a hypothesis space that was an intersection of an $l_{2}$ ball with a halfspace. It would be interesting to consider more general operational cost constraints, such as quadratic constraints and other convex functions. As it turns out, there are many applications where such constraints naturally arise. In current work, we are constructing bounds for these types of constraints, which lead to exotic hypothesis spaces, such as an intersection of an $l_{2}$ ball with an ellipsoid (for quadratic constraints) or a general convex body (for convex constraints).
\item \textit{Sequential MLOC:} Currently the MLOC framework applies to one-shot decision problems. It would be interesting to extend it to sequential decision problems, perhaps where multiple decisions are made in a sequence of decision epochs, and training data arrive incrementally. In this case, the baseline technique analogous to the ``sequential process" would be a Markov decision process (MDP). The MLOC framework would then assist in understanding the reasonable range of costs for various sequential decision policies.
\end{itemize}
}

\section{Conclusion}\label{sec:conclusion}
In this work, \textcolor{black}{we evaluated the MLOC framework in the context of a real application and demonstrated improvements over current standards. In particular,} we presented an application 
 in the domain of transportation routing called the ML\&TRP. 
 \textcolor{black}{Our framework }
 takes advantage of uncertainty in statistical modeling to explore the decision space and find potentially more practical solutions. 
\textcolor{black}{We provide experiments quantifying the improvements and the scalability of the framework with respect to routing problem size.}
We provided a generalization bound for the ML\&TRP (and for the general MLOC framework) indicating that a prior belief in the operational cost can potentially be beneficial to prediction ability in general. 

\bibliography{mloc_grid_bib}
\end{document}